\documentclass[11pt]{article}

\usepackage{bm}
\usepackage{amsfonts}
\usepackage{amssymb,amsmath,amsthm, amsfonts}
\usepackage{xcolor}
\usepackage{mathrsfs}
\usepackage{enumitem}

\def\sqr#1#2{{\vcenter{\vbox{\hrule height.#2pt
              \hbox{\vrule width.#2pt height#1pt \kern#1pt \vrule width.#2pt}
          \hrule height.#2pt}}}}
\def\sqr#1#2{{\vcenter{\vbox{\hrule height.#2pt
              \hbox{\vrule width.#2pt height#1pt \kern#1pt \vrule width.#2pt}
              \hrule height.#2pt}}}}
\def\3n{\negthinspace \negthinspace \negthinspace }
\def\2n{\negthinspace \negthinspace }
\def\1n{\negthinspace }

\def\={\buildrel \triangle \over =}

\def\liminf{\mathop{\underline{\rm lim}}}

\def\sup{\mathop{\rm sup}}

\def\({\Big (}
\def\){\Big )}
\def\[{\Big[}
\def\]{\Big]}
\def\be{\begin{equation}}

\def\ee{\end{equation}}

\def\square#1{\vbox{\hrule\hbox{\vrule height#1%
     \kern#1\vrule}\hrule}}
\def\rectangle#1#2{\vbox{\hrule\hbox{\vrule height#1%
     \kern#2\vrule}\hrule}}

\font\tenbb=msbm10 \font\sevenbb=msbm7 \font\fivebb=msbm5

\newfam\bbfam
\scriptscriptfont\bbfam=\fivebb \textfont\bbfam=\tenbb
\scriptfont\bbfam=\sevenbb

\newtheorem{lemma}{Lemma}[section]
\newtheorem{remark}{Remark}[section]

\newtheorem{theorem}{Theorem}[section]

\newtheorem{definition}{Definition}[section]
\newtheorem{proposition}{Proposition}[section]

\makeatletter
   
   \@addtoreset{equation}{section}
\makeatother

\begin{document}

\title{Stochastic Porous Media Equations With Nonhomogeneous Dirichlet Boundary Conditions}
\author{ Dan Goreac$^{1,2}$,\,\,Ioana Ciotir$^{3}$,\,\,Juan Li$^{2,4}$,\,\,Yufan Peng$^{2}$
\\
{$^1$\small \'Ecole d'actuariat, Universit\'e Laval, Qu\'ebec, QC, G1V 0A6, Canada.}\\
{$^2$\small School of Mathematics and Statistics, Shandong University, Weihai, Weihai 264209, P.~R.~China.}\\
{$^3$\small Normandie University, INSA de Rouen Normandie, LMI (EA 3226-FR CNRS 3335), }\\
\ \ \ \ \ \  {\small 76000 Rouen, France.}\\
{$^4$\small Research Center for Mathematics and Interdisciplinary Sciences, Shandong University,}\\
 \ \ \ \ \ \  {\small Qingdao 266237, P. R. China.}\\
{\footnotesize{\it E-mails:dan.goreac@univ-eiffel.fr,\,\ ioana.ciotir@insa-rouen.fr,\,\ juanli@sdu.edu.cn,\,\ yufanpeng@mail.sdu.edu.cn}}
}
\date{October 8, 2026}
\maketitle

\textbf{Abstract}.
We aim at studying the well-posedness of stochastic porous media equation problems under nonhomogeneous Dirichlet boundary. We obtain an unconventional homogeneous Dirichlet boundary stochastic porous medium equation by transformation under an additional assumption, and then use approximation problems to obtain the well-posedness properties of the equation with homogeneous Dirichlet boundary conditions under different initial conditions. Consequently, we obtain the corresponding solution to the original equation.


\section{Introduction}
In this paper we study nonlinear stochastic differential equations of the form
\begin{equation}\label{1.1}
\begin{cases}
dX(t)-\Delta\beta(X(t))dt=\sigma(X(t))dW(t),\quad\text{in}\ [0,+\infty)\times\mathcal{O},\\\\
X(t)=g(t),\quad\text{on}\ [0,+\infty)\times\partial\mathcal{O},\\\\
X(0)=x,\quad\text{in}\ \mathcal{O},
\end{cases}
\end{equation}
where $\mathcal{O}$ is an open bounded domain of $\mathbb{R}^d,\ d\geq 1$ with smooth boundary $\partial\mathcal{O}$, $\beta$ is a monotone function on $\mathbb{R}$, $W$ is a Wiener process and $\sigma=\sigma(X)$ is a suitable function to be made precise later on.

 The stochastic equation (\ref{1.1}) can be seen as an extension of deterministic equation when the forcing term  is replaced by a noise term. The interest of such questions for porous media type equations comes from geophysics and more precisely could be applied in the study and prevention of cliff erosion. According to the order of growth of the operator $\beta$, the porous media equation describes different phenomena going from slow diffusion, for over-unit orders, to fast diffusion, for sub-unit orders, and even super-fast diffusion for negative orders.

The stochastic porous media equation has been intensively studied recently in different frameworks. There is a vast literature on the subject and a complete overview would exceed the framework of this paper. Firstly, work was carried out in \cite{11, 13}, which respectively proved the existence of weak and strong solutions in specific situations. For more general cases see \cite{1}. Afterwards, the critical cases of fast and super-fast diffusion were discussed in \cite{19} and respectively \cite{superfast}, and a porous media Schrodinger type of equation was introduced and studied in \cite{schrodinger}. For an original perspective concerning uniqueness of the solution, see also \cite{francesco}. For the case of $\mathbb{R}^{d}$ see \cite{4} and \cite{reika}. For some recent developments involving a turbulence type of noise see \cite{div, stefan1,stefan2} and for the probabilistic representations see \cite{francesco2}. The above results are all studied under homogeneous Dirichlet boundary conditions. The case with Robin homogeneous boundary conditions was studies in \cite{neumann}. We treat the case of a nonhomogeneous Dirichlet boundary conditions. This causes significant changes in the solution space, resulting in the main difficulty in the proof. Since the solution space needs to be changed, we cannot directly reference previous results. We first translate the problem to a homogeneous case and then use Yosida approximations to consider an approximate equation, finally use approximation methods to obtain the solution of the homogeneous problem. After solving the homogeneous problem, we bring back the original problem to obtain the existence of strong solutions at $x \in L^2 $ and the existence of generalized solutions at $x\in H^{-1}(\mathcal O) $(The space $H^{-1}(\mathcal O)$ is introduced in Section 2).

More precisely, in this paper we study a stochastic porous medium equation on a smooth bounded domain $\mathcal O\subset\mathbb R^d$, $d\geq1$, with diffusion $\Delta\beta(X)$, noise term $\sigma(X)\,dW$, initial datum $x$, and prescribed nonhomogeneous Dirichlet boundary condition $X\big|_{[0, T]\times\partial\mathcal O}=g.$ The nonhomogeneous boundary condition is the main novelty addressed in this work. To deal with the boundary condition $g$ we consider an additional process $Y$ that attains the prescribed boundary value $g$ and has smooth properties. On setting $\bar{X}=X-Y$, we formulate a transformed problem with homogeneous boundary data for $\bar{X}$, interpreted according to the relevant notion of solution.

The transformed problem, whose nonlinearity depends on randomness and time via $Y$, requires estimates that incorporate the auxiliary process, rather than a direct application of existing results.

The organization of the paper is the following: we start with our introduction, in the second section we introduce various notations and spaces that will be used in the later section and then we will provide several lemmas that will be used later. In the third section, we will give the definition of strong solution and generalized solution, as well as an additional assumption. This assumption requires the existence of a $Y(t)$ that satisfies the boundary conditions and has sufficient other properties.

\section{Preliminaries}
Let $\mathcal{O}$ be a bounded open subset of $\mathbb{R}^{d}$. We assume that its boundary $\partial\mathcal{O}$ is sufficiently regular (at least of class $C^{2}$).\\
We shall introduce the following spaces:\\
$\bullet $ $L^{p}(\mathcal{O})=L^{p}$, $p\in[1,\infty],$  is the Banach space of all $p$-summable real-valued functions on $\mathcal{O}$ endowed with the usual norm $|\cdot|_{p}$. The inner product in the
Hilbert space $L^{2}(\mathcal{O})$  will be denoted by $\langle\cdot,\cdot\rangle_{2}$ .   \\
$\bullet $  $H^{k}(\mathcal{O})=H^{k}$, $k \in \mathbb{N}$ is the Sobolev space of all functions in $L^{2}$ whose distributional derivatives of order less than or equal to k belong to $L^{2}$. \\
$\bullet $ $H_{0}^{1}(\mathcal{O})=H_{0}^{1}$ is the set of all functions of $H^{1}$ that vanish on the boundary of $\mathcal{O}$. The norm in $H_{0}^{1}$ is denoted by $\|\cdot\|$ and given by    \\
\begin{equation*}
\|u\|_1:=\left(\int_{\mathbb{R}^d}|\nabla u(\xi)|^2d\xi\right)^{1/2}.
\end{equation*}
$\bullet $ $\mathcal{D}^{'}({\mathcal{O}})$ denotes the space of Schwartz distributions on $\mathcal{O}$.      \\
$\bullet $  $H^{-1}(\mathcal{O})=H^{-1}$ is the dual of $H_{0}^{1}(\mathcal{O})$. Its norm will be denoted by $\|\cdot\|_{-1}$ and the inner product by $\langle\cdot,\cdot\rangle_{-1}$.    \\
$\bullet $  $W^{k,p}(\mathcal{O})$  is the Sobolev space of all functions in $L^{p}$ whose
distributional derivatives of order less than or equal to $k$ belong to $L^{p}$.

Let $H$ be a separable Hilbert space and $p,\ q\in[1,+\infty]$ and let $(W_h)_{h\in\mathbb{N}}$
be a sequence of independent real $\mathcal{F}_{t}$-Brownian motions on a filtered probability space $(\Omega,\mathcal{F},\mathbb{P},(\mathcal{F}_t)_{t\geq 0})$ for some normal filtration $(\mathcal{F}_t)_{t\geq 0}$.\\
\noindent$\bullet $  $W$ is a cylindrical Wiener process in $H^{-1}$ to be used in the sequel.\\
\noindent$\bullet $  $L_W^q(0,T;L^p(\Omega,H))$ is the space of all $(\mathcal{F}_t)_{t\geq 0}$-progressively measurable processes $X$ such that
\begin{equation*}
\int_0^T(\mathbb{E}\left[|X(t)|^p\right])^\frac{q}{p}dt<\infty.
\end{equation*}
\noindent$\bullet $  $C_{W}([0,T];L^{p}(\Omega,H))$ is the space of all $H$-valued $(\mathcal{F}_t)_{t\geq 0}$-progressively measurable processes which are $p$ mean square continuous on $[0,T]$, that is such that
\begin{equation*}
\sup_{t\in[0,T]}\mathbb{E}\left[|X(t)|^p\right]<\infty.
\end{equation*}
\noindent$\bullet $  $L_{W}^{2}(\Omega;C([0,T];H))$ is the space of all $H$-valued $(\mathcal{F}_t)_{t\geq 0}$-progressively measurable processes which are continuous on $[0,T]$ and such that
\begin{equation*}
\mathbb{E}\left[\sup_{t\in[0,T]}|X(t)|^2\right]<\infty.
\end{equation*}

Let us consider an It\^{o}'s process with values in $H$ is a stochastic process $X(t)$, $t\in [0,T]$ of the form
\begin{equation}\label{2.3}
X(t)=x+\int_0^t b(s)ds+\int_0^t\sigma(s)dW(s),
\end{equation}
where $W$ is a cylindrical Wiener process in $H^{-1}$, $x\in H$, $b \in L_{\mathrm{W}}^{1}([0,T];L^{1}(\Omega,H))$ and $\sigma \in L_{W}^{2}([0,T];L^{2}(\Omega,\mathcal{L}_{2}(H))$. Here $\mathcal{L}_{2}(H)$ denotes the space of all Hilbert-Schmidt operators from $H$ to $H$. Then if $\varphi\in C^{2}(H)$ (i.e. twice continuously Fr\'{e}chet differentiable on H) the It\^{o}'s formula holds in the following form
\begin{equation}\label{2.4}
\begin{aligned}
\varphi(X(t))=&\varphi(x)+\int_{0}^{t}\langle b(s),D\varphi(X(s))\rangle ds+\frac{1}{2}\int_{0}^{t}\mathrm{Tr} [\sigma^{*}(s)D^{2}\varphi(X(s))\sigma(s)]ds\\
&+\int_{0}^{t}\langle D\varphi(X(s)),\sigma(s)dW(s)\rangle,
\end{aligned}
\end{equation}
where $D\varphi $ and $D^{2}\varphi $ represent the first and second derivatives of $\varphi$, respectively. Identity (\ref{2.4}) holds in $L_{W}^{2}(\Omega;C([0,T];H))$, P-a.s.

We note for further use that
\begin{equation}
\mathrm{Tr}\left[\sigma^*(s)D^2\varphi(X(s))\sigma(s)\right]=\sum_{k=1}^{\infty}\langle D^2\varphi(X(s))\sigma(s)g_k,\sigma(s)g_k\rangle,
\end{equation}\label{2.5}
where $\{g_{k}\}$ is an orthonormal basis of $H$.
Let $F\in L^{2}_{W}([0,T];L^2(\Omega,\mathcal{L}_2(H)))$ and set
\begin{equation*}
X(t):=\int^{t}_{0}F(s)dW(s),\ t\in[0,T].
\end{equation*}

Then $X$ belongs to $C_{W}([0,T];L^2(\Omega,H))$ and possesses a version which belongs
to $L^{2}_{W}(\Omega;C([0,T];H))$. Moreover, $X$ is a martingale and the following result holds:
\begin{proposition}{(Burkholder-Davis-Gundy)}
For arbitrary $p>0$ there exists a constant $c_{p}>0$, such that
\begin{equation}\label{2.6}
\mathbb{E}\left[\sup\limits_{t\leq T}\left|\int_0^tF(s)dW(s)\right|^p\right]\leq c_p\mathbb{E}\left[\left[\int_0^T\|F(s)\|_{\mathcal{L}_2(H)}^2ds\right]^{p/2}\right].
\end{equation}

By Doob's maximal inequality, we have
\begin{equation}\label{2.7}
\mathbb{E}\left[\sup\limits_{t\in[0,T]}\left|\int_0^tF(s)dW(s)\right|^2\right]\leq4\int_0^T\mathbb{E}\left[\|F(s)\|_{\mathcal{L}_2(H)}^2 \right] ds.
\end{equation}
\end{proposition}

\section{Assumptions}

For the below equation

\begin{equation}\label{3.1}
\begin{cases}
dX(t)=\Delta\beta(X(t))dt+\sigma(X(t))dW(t),\,\  t \in [0,T],\ \text{on } \mathcal{O},
\\X(0)=x\in H^{-1},
\\X(t)=g(t),\ \text{on}\ \Sigma =(0,T)\times \partial \mathcal{O}\times \Omega.
\end{cases}
\end{equation}
We assume that the following hypotheses are satisfied:

\textbf{(i)} $\beta:\mathbb{R}\to\mathbb{R}$ is monotonically increasing and $\beta(0)=0$. Moreover, there exists $K > 0$ such that
\begin{equation}\label{3.2}
|\beta(r)-\beta(s)|\leq K|r-s|,\quad\forall r,\ s\in\mathbb{R}.
\end{equation}

\textbf{(ii)}\begin{enumerate}[label=(\alph*)]\item $\sigma$ is Lipschitz from $H^{-1}$ into $\mathcal{L}_{2}(H^{-1})$, i.e. there exists $K_{1} > 0$ such that
\begin{equation}\label{3.3}
\|\sigma(x)-\sigma(y)\|_{\mathcal{L}_{2}(H^{-1},H^{-1})}\leq K_{1}\|x-y\|_{-1},\quad\forall x,\ y\in H^{-1},
\end{equation}
where $\mathcal{L}_2(H^{-1})$ denotes the space of all Hilbert-Schmidt operators in $H^{-1}$.

\item $\sigma$ is Lipschitzian from $L^{2}$ into $\mathcal{L}_{2}(H^{-1},L^{2})$, i.e. there exists $K_{2} > 0$ such that
\begin{equation}\label{3.4}
\|\sigma(x)-\sigma(y)\|_{\mathcal{L}_{2}(H^{-1},L^{2})}\leq K_{2}  |x-y |_{2},\quad\forall x,\ y\in L^{2},
\end{equation}
where $\mathcal{L}_2(H^{-1},L^2)$ denotes the space of all Hilbert-Schmidt operators from $H^{-1}$ to $L^{2}$.
\end{enumerate}
\begin{remark}
By assumption, we have
\begin{equation}\label{3.5}
\|\sigma(x)\|_{\mathcal{L}_{2}(H^{-1},H^{-1})}\leq\sigma_{1}+K_{1}\|x\|_{-1},\quad\forall x\in H^{-1},
\end{equation}
and
\begin{equation}\label{3.6}
\|\sigma(x)\|_{\mathcal{L}_{2}(H^{-1},L^{2})}\leq\sigma_{2}+K_{2} | x |_{2},\quad\forall x\in L^{2},
\end{equation}
where $\sigma_{2}=|\sigma(0)|_{\mathcal{L}_{2}(H^{-1},L^{2})}<\infty$ and $\sigma_{2}=|\sigma(0)|_{\mathcal{L}_{2}(H^{-1},L^{2})}<\infty$.
\end{remark}

\textbf{(iii)} Concerning the regularity of $g$, we take it smooth enough such that there exists a process
$$Y\in L_{W}^{2}(\Omega;C([0,T];H^1))\cap L^{2}(\Omega;L^{\infty}([0,T]\times\mathcal{O} )),$$
satisfying:
\begin{equation}\label{3.7}
\begin{cases}
Y(t)=g(t)\  \text{on}\ \Sigma ,\\
Y(0)=y\in L^{2}\ ,\\
dY(t)=b(t)dt.
\end{cases}
\end{equation}
Where $$b(t)\in L^{2}(0,T;L^{2}(\Omega;L^{2})),$$
and such that
$$\Delta\beta(Y(t))\in L^{2}(0,T;L^2(\Omega,L^2)).$$

\section{Approach}

In this work we are interested to prove existence and uniqueness of the solution in the following sense.

\begin{definition}
For $x\in H^{-1}$, a strong solution to (\ref{3.1}) on $[0,T]$ is a stochastic process $X$ which satisfies:
\begin{equation*}
\begin{aligned}
&X\in L_W^2(\Omega;C([0,T];H^{-1}))\cap L^2(0,T;L^2(\Omega;L^2));\\
&\beta(X)\in L^2(0,T;L^2(\Omega;H^1));\\
&X(t)=x+\int_0^t\Delta\beta(X(s))ds+\int_0^t\sigma(X(s))dW(s),\quad\forall t\in[0,T],\mathbb{P}\text{-a.s.}
\end{aligned}
\end{equation*}
It follows that a strong solution of (\ref{3.1}) is an It\^{o} process in $H^{-1}$.

For $x\in H^{-1}$ a generalized solution to (\ref{3.1}) on $[0,T]$ is a process X which belongs to $L_{W}^{2}(\Omega;C([0,T];H^{-1}))$, and such that there exists a sequence $\{x_n\}\subset L^2$ convergent to $x$ with the topology on $H^{-1}$ such that
\begin{equation*}
\lim\limits_{n\to\infty}X_n=X\quad\text{in}\ L_W^2(\Omega;C([0,T];H^{-1})),
\end{equation*}
where $X_n$ is the strong solution to (\ref{3.1}) with $x_n$ replacing $x$.

\end{definition}

Let us introduce
\begin{equation*}
\begin{cases}
\bar{X}(t)=X(t)-Y(t),\\
F(x)=\beta(x+Y(t))-\beta(Y(t)),\\
\bar{\sigma}(x)=\sigma(x+Y(t)),\\
f(t)=\Delta\beta(Y(t))-b(t).
\end{cases}
\end{equation*}
And we rewrite the equation above in the following form

\begin{equation}\label{3.8}
\begin{cases}
d\bar{X}(t)=\Delta F(\bar{X}(t))dt+f(t)dt+\bar{\sigma}(\bar{X}(t))dW(t),\ t\geq0,
\\ \bar{X}(0)=(x-y)\in H^{-1} ,
\\ \bar{X}(t)=0, \ \ \ \ \ \ \ \ \ \text{on}\ \Sigma.
\end{cases}
\end{equation}

From the previous assumption, it is easy to check that:

\textbf{(i)} $F:[0,T] \times \Omega\times \mathcal{O} \times \mathbb{R} \to\mathbb{R}$, $F(t,\omega,\xi,\cdot )$ is monotonically increasing and $F(t,\omega,\xi,0)=0$.
Moreover, there exists $K > 0$ such that $F$ is Lipschitz about $x$ uniformly in $(t,\omega,\xi)$, i.e.
\begin{equation}\label{3.9}
|F(t,\omega,\xi,r)-F(t,\omega,\xi,s)|\leq K|r-s|,\quad\forall r,\ s\in\mathbb{R}.
\end{equation}

\textbf{(ii)}
\begin{enumerate}[label=(\alph*)]
\item $\bar{\sigma}(t,\omega,\xi,\cdot )$\ is Lipschitz continuous from $H^{-1}$ into $\mathcal{L}_{2}(H^{-1})$, there exists $K_{1} > 0$ such that
\begin{equation}\label{3.10}
\|\bar{\sigma}(t,\omega,\xi,x)-\bar{\sigma}(t,\omega,\xi,y)\|_{\mathcal{L}_{2}(H^{-1},H^{-1})}\leq K_{1}\|x-y\|_{-1},\quad\forall x,\ y\in H^{-1},
\end{equation}
and we have
\begin{equation}\label{3.11}
\|\bar{\sigma}(t,\omega,\xi,x)\|_{\mathcal{L}_{2}(H^{-1},H^{-1})}\leq \sigma_{1}+K_{1}\| Y(t,\omega)\|_{-1}+K_{1}\|x\|_{-1},\quad\forall x\in H^{-1},
\end{equation}
let $\bar{\sigma}_{1}=\sigma_{1}+K_{1}\| Y(t,\omega)\|_{-1}$ and $\mathbb{E}[\bar{\sigma}^2_1]\leq 2\sigma^2_1+2 K_1\mathbb{E}[\sup\limits_{t\in [0,T]}\| Y(t,\omega)\|^2_{-1}]< \infty$.

\item $\bar{\sigma}(t,\omega,\xi,\cdot )$\ is Lipschitz continuous from $L^{2}$ into $\mathcal{L}_{2}(H^{-1},L^{2})$, there exists $K_{2} > 0$ such that
\begin{equation}\label{3.12}
\|\bar{\sigma}(t,\omega,\xi,x)-\bar{\sigma}(t,\omega,\xi,y)\|_{\mathcal{L}_{2}(H^{-1},L^{2})}\leq K_{2}  |x-y |_{2},\quad\forall x,\ y\in L^{2},
\end{equation}
and we have
\begin{equation}\label{3.13}
\|\bar{\sigma}(t,\omega,x)\|_{\mathcal{L}_{2}(H^{-1},L^{2})}\leq \sigma_{2}+K_2|Y(t,\omega)|_2+K_{2} | x |_{2},\quad\forall x\in L^{2},
\end{equation}
let $\bar{\sigma}_{2}=\bar{\sigma}_2+K_2|Y(t,\omega)|_2$ and $\mathbb{E}[\bar{\sigma}^2_2]\leq 2\sigma^2_2+2\mathbb{E}[K_2\sup\limits_{t\in [0,T]}|Y|^2_2]< \infty$.
\end{enumerate}
If there is no danger of confusion, we omit the superscript $(t,\omega,\xi)$ and replace $F(t,\omega,\xi,\cdot)$ with $F(\cdot)$.

\section{Uniqueness of the solution}

\begin{proposition}\label{prop3.1}
Equation (\ref{3.8}) has at most one strong solution in the sense of the definition above.
\end{proposition}

\begin{proof}
Let $X_{1}$ and $X_{2}$ be two strong solutions to (\ref{3.8}). Then, by It\^{o} 's formula, we
have
\begin{equation*}
\begin{aligned}
d\|X_{1}(t)-X_{2}(t)\|^{2}_{-1}=&2\langle X_{1}(t)-X_{2}(t),\Delta F(X_{1}(t))- \Delta F(X_{2}(t))\rangle_{-1} dt\\
&+2\langle X_{1}(t)-X_{2}(t),(\bar{\sigma}(X_{1}(t))- \bar{\sigma}(X_{2}(t)))dW(t)\rangle_{-1}\\
&+\| \bar{\sigma}(X_{1}(t))-\bar{\sigma}(X_{2}(t))   \|^2_{\mathcal{L}_{2}(H^{-1})}dt,\  \mathbb{P}-a.s.\\
\end{aligned}
\end{equation*}
Therefore,
\begin{equation*}
\begin{aligned}
\mathbb{E}\left[\|X_{1}(t)-X_{2}(t) \|^{2}_{-1} \right]=&\mathbb{E}\left[\int^{t}_{0} 2\langle X_{1}(s)-X_{2}(s),\Delta F(X_{1}(s))- \Delta F(X_{2}(s))\rangle_{-1} ds\right]\\
&+\mathbb{E}\left[\int^{t}_{0}\| \bar{\sigma}(X_{1}(s))-\bar{\sigma}(X_{2}(s))   \|^2_{\mathcal{L}_{2}(H^{-1})}ds\right],\\
\end{aligned}
\end{equation*}
for $x,\ {x_{1}}\in D(F)$, $D(F)$ is the domain of $F$,
\begin{equation*}
\begin{aligned}
\langle \Delta F(x)-\Delta F(x_{1}),x-x_{1}\rangle_{-1}=\langle-F(x)+F(x_{1}),x-x_{1}\rangle_{2} \leq 0,
\end{aligned}
\end{equation*}
then by (\ref{3.10}) we have
\begin{equation*}
\begin{aligned}
\mathbb{E}[\|X_{1}(t)-X_{2}(t) \|^{2}_{-1} ] \leq K^2_{1}\int^{t}_{0}\mathbb{E}[\|X_{1}(s)-X_{2}(s) \|^{2}_{-1} ]ds.
\end{aligned}
\end{equation*}
Now the conclusion follows from Gronwall's lemma.
\end{proof}

\section{Existence of the solution}

To prove the existence, let us approximate (\ref{3.8}) by an equation of the form
\begin{equation}\label{4.1}
\begin{cases}
dX_\epsilon+A_\epsilon(X_\epsilon)dt=f(t)dt+\bar{\sigma}(X_\epsilon)dW(t),\\\\
X_\epsilon(0)=x.
\end{cases}
\end{equation}
Here $A_{\epsilon}$ is the Yosida approximation in $H^{-1}$ of $G_{\epsilon}$ given as follows

\begin{equation}\label{4.2}
G^{(t,\omega)}_\epsilon(u):=-\Delta(F(t,\omega,u)+\epsilon I)(u),\quad\forall u\in H_0^1,
\end{equation}
for every $(t,\omega)$ is fixed.

To this purpose, we need the following proposition.

\begin{proposition}
$G^{(t,\omega)}_\epsilon(x)$ is maximal monotone. Moreover there exists a function $\phi$ is convex and lower semi-continuous on $H^{-1}(\mathcal{O})$ such that
\begin{equation*}
\partial \phi (u) =-\Delta ((F(t,\omega,u)+\epsilon I)(u),\quad\forall u\in H_0^1,
\end{equation*}
for every $(t,\omega)$ is fixed.
\end{proposition}

\begin{proof}

We will first set $$j(r)=\int^{r}_{0}(\beta(s+Y)-\beta(Y)+\epsilon s)ds.$$

Then we can easily get $j:[0,T]\times \Omega \times \mathcal{O}\times \mathbb{R}\rightarrow \mathbb{\bar{R}}$ is an continuous, convex, proper function in the last variable.

We can now define the function $\phi:H^{-1}(\mathcal{O})\rightarrow \mathbb{\bar{R}}$ by
\begin{equation*}
\left.\phi(u)=
\begin{cases}\int_{\mathcal{O}}j(u(x))dx,&\mathrm{if}\ u\in L^1(\mathcal{O})\ \mathrm{and}\ j(u)\in L^1(\mathcal{O});\\+\infty,&\mathrm{otherwise}.\end{cases}\right.
\end{equation*}

\textbf{Step I.}
We will prove that the function $\phi$ is convex and lower semi-continuous on $H^{-1}(\mathcal{O})$.

To this purpose, we need first to show that $$\lim_{|r|\rightarrow \infty}\frac{j(r)}{|r|}=+\infty.$$

Since $Y\in L^{2}(\Omega;L^{\infty}([0,T]\times\mathcal{O} ))$ and $\beta$ is monotonically, we have
\begin{equation}\label{4.3}
\begin{aligned}
\lim_{|r|\rightarrow \infty}\frac{j(r)}{|r|}&=\lim_{|r|\rightarrow \infty}\frac{\int^{r}_{0}(\beta(s+Y)-\beta(Y)+\epsilon s)ds}{|r|}\\
&=\lim_{|r|\rightarrow \infty}\frac{\int^{r}_{0}(\beta(s+Y)-\beta(Y))ds}{|r|}+\frac{\int^{r}_{0}\epsilon sds}{|r|}\\
&\geq\lim_{|r|\rightarrow \infty}\frac{\int^{r}_{0}\epsilon sds}{|r|}\\
&=\lim_{|r|\rightarrow \infty}\frac{\epsilon|r|}{2}=+\infty.
\end{aligned}
\end{equation}

From the fact that $j$ is convex  in the last variable we deduce immediately that $\phi$ is also convex.

Concerning the lower semicontinuity, we need first to study it in $L^1(\mathcal{O})$. By the continuity of $j$, let  $\{ u_n \}\subset L^{1}(\mathcal{O})$ be strongly convergent to $u$. Then there is $\{ u_{n_k} \}\subset \{ u_n \}$ such that $u_{n_k}(x)\rightarrow u(x)$ a.e. $x\in \mathcal{O}$ for $n_k\rightarrow \infty$. Then by Fatou's lemma, it follows that
\begin{equation*}
\liminf_{n_k \to \infty}\int_{\mathcal{O}}j(x,u_{n_k}(x))dx\geq \int_{\mathcal{O}}j(x,u(x))dx,
\end{equation*}
there $\liminf_{n_k \to \infty}\phi(u_{n_k})\geq \phi(u)$. That implies that $\liminf_{n \to \infty}\phi(u_{n})\geq \phi(u)$, i.e. $\phi$ is l.s.c in the last variable in $L^1(\mathcal{O})$.

Now we consider a sequence $\{u_n\}\subset H^{-1}(\mathcal{O})\cap L^1(\mathcal{O})$ such that $u_n\rightarrow u$ in $H^{-1}(\mathcal{O})$ and $\phi(u_n)\leq \lambda$, i.e. $\int_{\mathcal{O}}j(u_n)dx\leq \lambda,\ \forall n$. We need prove that $\int_{\mathcal{O}}j(u)dx\leq \lambda$. Since this function is convex, it is weakly lower semicontinuous in $L^1(\mathcal{O})$ and so it suffices to show that $\{ u_n \}$ is weakly compact in $L^1(\mathcal{O})$. According to Dunford-Pettis criterion, we must prove that $\{ u_n \}$ is bounded in $L^1(\mathcal{O})$ and $\int_{\mathcal{O}}|u_n|d\xi$ are uniformly absolutely continuous, i.e. for any $ \varepsilon >0$, exist $\delta(\varepsilon)$ such that $\int_{A}|u_n(\xi)|d\xi\leq \varepsilon$ if $m(A)\leq \delta(\varepsilon)$ where $A$ is a measurable set of $\Omega$. By (\ref{4.3}), for every $p>0$ there exists $R(p)>0$ such that $j(r)\geq p|r|$ if $|r|\geq R(p)$. This implies that $\int_{\mathcal{O}}|u_n(\xi)|d\xi\leq C$, where $C>0$ is a constant, since
\begin{equation*}
\begin{aligned}
\int_{\mathcal{O}}|u_n(\xi)|d\xi&\leq \int_{\mathcal{O}\cap \{ u_n(\xi)\geq R(p)\}}|u_n(\xi)|d\xi+\int_{\mathcal{O}\cap \{ u_n(\xi)\leq R(p)\}}|u_n(\xi)|d\xi\\
&\leq \frac{\lambda}{p}+R(p)m(\mathcal{O}).
\end{aligned}
\end{equation*}
Moreover, for every measurable subset $E$ of $\mathcal{O}$, we have
\begin{equation*}
\begin{aligned}
\int_E|u_n(\xi)|d\xi &\leq \int_{E\cap \{ u_n(\xi)\geq R(p)\}}|u_n(\xi)|d\xi+\int_{E\cap \{ u_n(\xi)\leq R(p)\}}|u_n(\xi)|d\xi\\
&\leq \frac{1}{p}\int_{\mathcal{O}}|u_n(\xi)|d\xi +R(p)m(E).
\end{aligned}
\end{equation*}
If we choose $p>(2\varepsilon)^{-1}\sup\int_{\mathcal{O}}|u_n(\xi)|d\xi$ and $m(E)\leq \frac{\varepsilon}{2R(p)}$. Hence, $\{ u_n \}$ is weakly compact in $L^1(\mathcal{O})$. We get that the function is convex and lower semicontinuous.
By Theorem 2.1 of \cite{7}, we know that for $\phi:X\to \bar{\mathbb{R}} $ which is a l.s.c. proper convex function we have that $\partial\phi$ is a maximal monotone operator.

\vspace{10pt}

\textbf{Step II.}
We will prove now that $G_\epsilon$ is also a maximal monotone in $H^{-1}(\mathcal{O})\cap H^{-1}(\mathcal{O})$.

Let $f$ be arbitrary but fixed in $H^{-1}(\mathcal{O})$. We must show that there exist $u\in H^{-1}(\mathcal{O})\cap L^1(\mathcal{O})$ and $v\in H^1_0(\mathcal{O})$ such that
\begin{equation*}
u-\Delta (F+\epsilon I)(u)=f\, \, \text{in}\,  \mathcal{O},
\end{equation*}
equivalently
\begin{equation}\label{4.5}
\gamma(v)-\Delta v=f\, \, \text{in}\, \mathcal{O},
\end{equation}
where $\gamma(v)=(F+\epsilon I)^{-1}(v)$ and $\gamma(v)\in H^{-1}(\mathcal{O})\cap L^1(\mathcal{O}),\ v\in H^1_0(\mathcal{O})$.

Consider the Yosida approximate equation,
\begin{equation}\label{4.6}
\gamma_{\lambda}(v_{\lambda})-\Delta v_{\lambda}=f\ \ \text{in}\  \mathcal{O},\ \ v_{\lambda}=0\ \text{in}\ \partial\mathcal{O},
\end{equation}
where $\gamma_{\lambda}=\lambda^{-1}(1-(1+\lambda\gamma)^{-1}),\ \lambda >0$ is the Yosida approximation of $\gamma_{\lambda}$ in $\mathbb{R}$. It is readily seen that (\ref{4.6}) has a unique solution $v_{\lambda}\in H^1_0(\mathcal{O})$. Indeed, since $-\Delta$ is maximal monotone from $H^1_0(\mathcal{O})$ to $H^{-1}(\mathcal{O})$ and $v\rightarrow \gamma_{\lambda}(v)$ is monotone and continuous from $H^1_0(\mathcal{O})$ to $H^{-1}(\mathcal{O})$ we infer by Corollary 2.1 of \cite{8} that $v\rightarrow \gamma_{\lambda}(v)-\Delta v$ is maximal monotone in $H^1_0(\mathcal{O})\times H^{-1}(\mathcal{O})$, and by Corollary 1.2 of \cite{7} that it is surjective.

Multiplying (\ref{4.6}) by $v_{\lambda}$ and by integrating over $\mathcal{O}$, we get
\begin{equation*}
\int_{\mathcal{O}}|\nabla v_{\lambda}|^2 d\xi+\int_{\mathcal{O}}v_{\lambda}\gamma_{\lambda}(v_{\lambda}) d\xi=\int_{\mathcal{O}}fv_{\lambda}d\xi.
\end{equation*}
Because $\gamma_{\lambda}$ is monotone and,
\begin{equation*}
\int_{\mathcal{O}}|\nabla v_{\lambda}|^2 d\xi \leq \frac{1}{4c_{\mathcal{O}}^2}\| f \|^2_{-1},
\end{equation*}
where $c_{\mathcal{O}}$ depending only on $\mathcal{O}$ and on the dimension $N$.

Hence, $v_{\lambda}$ is bounded in $H^1_0(\mathcal{O})$. Then, on a subsequence, again denoted $\lambda$, we have
\begin{equation*}
v_{\lambda}\rightharpoonup v\ \text{in}\ H^1_0(\mathcal{O}),\ \ v_{\lambda}\rightharpoonup v\  \text{in}\ L^2(\mathcal{O}).
\end{equation*}
Thus, extracting further subsequences, we may assume that
\begin{equation}\label{4.7}
\begin{aligned}
v_{\lambda}(x)\rightarrow v(x)&\ \ \ a.e.\ x\in\mathcal{O},\\
(1+\lambda\gamma)^{-1}v_{\lambda}(x)\rightarrow v(x)&\ \ \ a.e.\ x\in\mathcal{O},
\end{aligned}
\end{equation}
because for arbitrary but fix $x\in \mathcal{O}$, we have $D(\gamma)=R(F+\epsilon I)=\mathbb{R}$, and
\begin{equation*}
\|v_{\lambda}-(1+\lambda\gamma)^{-1}v_{\lambda} \|^2_2\leq\lambda\| \gamma(1+\lambda\gamma)^{-1}v_{\lambda} \|^2_2\leq\frac{\lambda}{\epsilon}\| v_{\lambda} \|^2_2\rightarrow 0,
\end{equation*}
when $\lambda\rightarrow0$.

Then, letting $\lambda$ tend to zero in (\ref{4.6}), we see that $\gamma_{\lambda}(v_{\lambda})\rightarrow u$ in $H^{-1}(\mathcal{O})$ and
\begin{equation*}
u-\Delta v=f\ \ \text{in}\  \mathcal{O},\ \ v(x)\in H^1_0(\mathcal{O}).
\end{equation*}
It remains to be shown that $u\in L^1(\mathcal{O})$ and $u(x)= \gamma(v(x))$ a.e.$x\in \mathcal{O}$.

Multiplying (\ref{4.6}) by $v_{\lambda}$, we see that
\begin{equation*}
\int_{\mathcal{O}}\gamma_{\lambda}(v_{\lambda})v_{\lambda}dx\leq \frac{1}{4}\|f\|^2_{-1}.
\end{equation*}
On the other hand, for some $u_0\in D(j)$, because of $\partial_r j=F+\epsilon I$, we have $j(\gamma_{\lambda}(v_{\lambda}))\leq j(u_0)+(\gamma_{\lambda}(v_{\lambda})(x)-u_0)v$ for $v \in (F+\epsilon I)(\gamma_{\lambda}(v_{\lambda}))$. This yields
\begin{equation*}
\begin{aligned}
\int_{\mathcal{O}}j(g_{\lambda}(x))dx\leq& \int_{\mathcal{O}} j(u_0) dx+\int_{\mathcal{O}}\gamma_{\lambda}(v_{\lambda})(x)(1+\lambda\gamma)^{-1}v_{\lambda}  dx+\frac{1}{2}\int_{\mathcal{O}} |u_{0}|^2+|v_{\lambda}|^2 dx\\
\leq& \int_{\mathcal{O}} j(u_0) dx+\int_{\mathcal{O}}\gamma_{\lambda}(v_{\lambda})(x)(1+\lambda\gamma)^{-1}v_{\lambda}  dx\\
&+\int_{\mathcal{O}}  \gamma_{\lambda}(v_{\lambda})(x)\lambda\gamma(1+\lambda\gamma)^{-1}v_{\lambda} dx+\frac{1}{2}\int_{\mathcal{O}} |u_{0}|^2+|v_{\lambda}|^2 dx\\
\leq& \int_{\mathcal{O}} j(u_0) dx+\int_{\mathcal{O}}\gamma_{\lambda}(v_{\lambda})(x)v_{\lambda}  dx+\frac{1}{2}\int_{\mathcal{O}} |u_{0}|^2+|v_{\lambda}|^2 dx\leq C,
\end{aligned}
\end{equation*}
because $(1+\lambda\gamma)^{-1}v_{\lambda}= (F+\epsilon I)(\gamma_{\lambda}(v_{\lambda}))$ and $\gamma_{\lambda}(v_{\lambda})=\gamma(1+\lambda\gamma)^{-1}v_{\lambda}$. As seen before, this implies that $\{ \gamma_{\lambda}(v_{\lambda}) \}$ is weakly compact in $L^1(\mathcal{O})$. Hence, $u\in L^1(\mathcal{O})$ and
\begin{equation}\label{4.8}
\gamma_{\lambda}(v_{\lambda})\rightharpoonup u\ \ \ \text{in}\ L^1(\mathcal{O})\ \ \ \ \text{for}\ \lambda\rightarrow 0.
\end{equation}
On the other hand, by (\ref{4.7}) it follows by virtue of the Egorov theorem that for every $\varepsilon >0$ there exists a measurable subset $E_{\varepsilon}\in \mathcal{O}$ such that $m(\mathcal{O}\backslash E_{\varepsilon})\leq \varepsilon,\ \{ (1+\lambda\gamma)^{-1}v_{\lambda} \}$ is bounded in $L^{\infty}(E_{\varepsilon})$, and
\begin{equation}\label{4.9}
(1+\lambda\gamma)^{-1}v_{\lambda}\rightarrow v\ \ \text{uniformly in}\ E_{\varepsilon}\ \ \ \text{as}\ \lambda\rightarrow0.
\end{equation}
Since $\varepsilon$ is arbitrary, we obtain by classical arguement that $v(x)= \gamma(u(x))$ a.e. $x\in \mathcal{O}$, as desired.

\vspace{10pt}

\textbf{Step III.}
We finally prove that $G_\epsilon=\partial \phi$.

We consider an arbitrary element $[u, -\Delta v]\in G_\epsilon$, i.e.,\ $u\in H^{-1}(\mathcal{O})\cap L^{1}(\mathcal{O})$, $v\in H^1_0(\mathcal{O})$, $v(x)=(F+\epsilon I)(u(x))$ a.e.\ $x\in \mathcal{O}$. Since $v(x)(u(x)-\bar{u}(x))\geq j(u(x))-j(\bar{u}(x))$ a.e. $x\in \mathcal{O}$, it follows by Lemma 2.2 of \cite{7} that
\begin{equation*}
\langle G_\epsilon u,u-\bar{u} \rangle=(v,u-\bar{u})=\int_{\mathcal{O}}v(x)(u(x)-\bar{u}(x))dx\geq\int_{\mathcal{O}}j(u(x))-j(\bar{u}(x))dx,\ \forall\bar{u}\in D(\phi).
\end{equation*}
Hence,
\begin{equation*}
\langle G_\epsilon u,u-\bar{u} \rangle\geq \phi(u)-\phi(\bar{u}),\ \ \ \forall\bar{u}\in H^{-1}(\mathcal{O}),
\end{equation*}
this implies $G_\epsilon(u)\subset \partial \phi(u)$. We complete the proof.
\end{proof}

Now we consider $A_\epsilon$ the Yosida approximation of $G_\epsilon$, i.e.
\begin{equation}\label{4.10}
A_\epsilon(t,\omega,x):=\frac{1}{\epsilon}(x-J^{(t,\omega)}_\epsilon(x))=-\Delta(F(t,\omega,\cdot)+\epsilon I)(J^{(t,\omega)}_\epsilon(x)),\quad x\in H^{-1},
\end{equation}
where $J^{(t,\omega)}_{\epsilon}$ is the
\begin{equation}\label{4.11}
J^{(t,\omega)}_{\epsilon}:=(I+\epsilon G^{(t,\omega)}_{\epsilon})^{-1},\quad\epsilon>0.
\end{equation}
We have that $A_{\epsilon}=-\Delta(F+\epsilon I)(J_\epsilon(x))$ is Lipschitz continuous both in $H^{-1}$ and in $L^{2}$ therefore that problem (\ref{4.1}) has a unique strong solution $X_{\epsilon}$ belonging to $L_{W}^{2}(\Omega;C([0,T];H^{-1}))$ for all $x \in H^{-1}$ and a unique strong solution $X_{\epsilon}$ belonging to $L_{W}^{2}(\Omega;C([0,T];L^{2}))$ for all $x \in L^{2}$ by standard existence results.

In order to pass to the limit as $\epsilon\rightarrow 0$ we need to check the following properties of the resolvent.
\begin{lemma}\label{lem3.1}
We have that
\begin{align}
\mathrm{i})\ &\|J_{\epsilon}(x)-J_{\epsilon}(x_{1})\|_{-1}\leq\|x-x_{1}\|_{-1},\quad\forall x,x_{1}\in H^{-1}, \label{4.12} \\
&\|J_\epsilon(x)\|_{-1}\leq\|x\|_{-1},\quad\forall x\in H^{-1}, \label{4.13} \\
\mathrm{ii})\ &|J_{\epsilon}(x)-J_{\epsilon}(x_{1})|_{2}\leq\frac{c}{\epsilon}|x-x_{1}|_{2},\quad\forall x,x_{1}\in L^{2}, \label{4.14} \\
\mathrm{iii})\ &|J_\epsilon(x)|_2\leq|x|_2+\epsilon |\Delta\beta(Y)|_2+\sqrt{\frac{\epsilon K}{2}}\left(\int_{\mathcal{O}} |\nabla Y|^2 d\xi\right)^{\frac{1}{2}},\quad\forall x\in L^2, \label{4.15} \\
\mathrm{iv})\ &\mathbb{E}\left[\int^{T}_{0}\| f(t) \|^{2}_{-1}dt\right]< \infty. \label{4.16}
\end{align}
\end{lemma}
\begin{proof}

\vspace{10pt}

$\mathrm{i})$ Let $x$, $x_{1}\in H^{-1}$, $y=J_{\epsilon}(x)$, $y_{1}=J_{\epsilon}(x_{1})$
\begin{equation}\label{4.17}
(y-y_{1})-\epsilon(\Delta(F(t,\omega,\cdot)+\epsilon I)(y)-\Delta(F(t,\omega,\cdot)+\epsilon I)(y_{1}))=x-x_{1},
\end{equation}
multiply by $(y-y_{1})$ ,because $-\epsilon\Delta(F(t,\omega,\cdot)+\epsilon I)$ is monotone, one gets
\begin{equation*}
\|y-y_{1}\|_{-1}^{2}\leq\|y-y_{1}\|_{-1}\|x-x_{1}\|_{-1}.
\end{equation*}
This inplies we get (\ref{4.12}). Set $x_{1}=0$ to get (\ref{4.13}).

\vspace{10pt}

$\mathrm{ii})$ Let $x,\ x_{1}\in H^{-1}$, $y=J_{\epsilon}(x)$, $y_{1}=J_{\epsilon}(x_{1})$. One has
\begin{equation}\label{ast}
y-y_{1}-\epsilon\Delta(F(t,\omega,\cdot)+\epsilon I)(y)+\epsilon\Delta(F(t,\omega,\cdot)+\epsilon I)(y_{1})=x-x_{1}.
\end{equation}
Recall that
\begin{equation*}
\langle F(F+\epsilon I)^{-1}(r)-F(F +\epsilon I)^{-1}(s),(F +\epsilon I)^{-1}(r)-(F +\epsilon I)^{-1}(s)\rangle \geq0,
\end{equation*}
it follows that
\begin{equation*}
\begin{aligned}
&\langle r-s,(F +\epsilon I)^{-1}(r)-(F +\epsilon I)^{-1}(s)\rangle \\
\geq& \epsilon \langle (F+\epsilon I)^{-1}(r)-(F +\epsilon I)^{-1}(S),(F +\epsilon I)^{-1}(r)-(F +\epsilon I)^{-1}(s)\rangle.
\end{aligned}
\end{equation*}
As a consequence
\begin{equation*}
\begin{aligned}
\vert(F+\epsilon I)^{-1}(r)-(F+\epsilon I)^{-1}(s)\vert\leq \frac{1}{\epsilon}\vert r-s\vert.
\end{aligned}
\end{equation*}
Since $(F+\epsilon I)(y)\in H^{1}_{0}$, we have $y\in H^{1}_{0}$. Multiplying on both sides of (\ref{ast}) by $(F+\epsilon I)(y)-(F+\epsilon I)(y_{1})$ and by integrating over $\mathcal{O}$,
\begin{equation*}
\begin{aligned}
&\int_{\mathcal{O}}(y-y_{1})((F+\epsilon I)(y)-(F+\epsilon I)(y_{1}))d\xi\\
&-\epsilon\langle(\Delta(F+\epsilon I)(y)-\Delta(F+\epsilon I)(y_{1})),((F+\epsilon I)(y)-(F+\epsilon I)(y_{1}))\rangle  \\
=&\int_{\mathcal{O}}(x-x_{1})((F+\epsilon I)(y)-(F+\epsilon I)(y_{1}))d\xi.
\end{aligned}
\end{equation*}
Taking into account that $F$ is monotone increasing it follows that
\begin{equation*}
\epsilon|y-y_{1}|_2^2+\epsilon\|(F+\epsilon I)(y)-(F+\epsilon I)(y_{1})\|_1^2\leq|x-x_{1}|_2|(F+\epsilon I)(y)-(F+\epsilon I)(y_{1})|_2.
\end{equation*}
Finally, by Poincar\'{e}'s inequality there exists $C>0$ such that
\begin{equation*}
\begin{aligned}
&\epsilon|y-y_{1}|_{2}^{2}+\epsilon C|(F+\epsilon I)(y)-(F+\epsilon I)(y_{1})|_{2}^{2} \\
\leq&|x-x_{1}|_2|(F+\epsilon I)(y)-(F+\epsilon I)(y_{1})|_2 \\
\leq&\epsilon C|(F+\epsilon I)(y)-(F+\epsilon I)(y_{1})|_2^2+\frac{1}{4\epsilon C}|x-x_{1}|_2^2,
\end{aligned}
\end{equation*}
and therefore
\begin{equation*}
|y-y_{1}|_{2}^{2}\leq\frac{1}{4\epsilon^{2}C}|x-x_{1}|_{2}^{2},
\end{equation*}
such that (\ref{4.14}) is proved.

\vspace{10pt}

$\mathrm{iii})$ Let $y_{1}=0$ and fix $x\in H^{-1}$, we have
\begin{equation}\label{asst}
y-\epsilon\Delta(F(t,\omega,\cdot)+\epsilon I)(y)=x.
\end{equation}
One obviously has $0\leq\int_{\mathcal{O}}\nabla\beta(y+Y)\cdot\nabla\beta(y+Y) d\xi$.
\begin{equation*}
\begin{aligned}
&\int_{\mathcal{O}}\nabla\beta(y+Y)\cdot\nabla\beta(y+Y)d\xi\\
=&\int_{\mathcal{O}}\sum^{d}_{i=1}(\frac{\partial}{\partial x_i}\beta(y(\xi)+Y(\xi)))^{2}d\xi\\
\leq& \int_{\mathcal{O}} K\sum^{d}_{i=1}(\frac{\partial}{\partial x_i}(y(\xi)+Y(\xi)))(\frac{\partial}{\partial x_i}\beta(y(\xi)+Y(\xi)))d\xi\\
=&K\int_{\mathcal{O}}\nabla\beta(y+Y)\cdot\nabla (y+Y)d\xi\\
\leq& K\int_{\mathcal{O}} \nabla\beta(y+Y)\cdot \nabla y d\xi+\frac{K^2}{2}\int_{\mathcal{O}} |\nabla Y|^2 d\xi+ \frac{1}{2}\int_{\mathcal{O}} |\nabla\beta(y+Y)|^2 d\xi.
\end{aligned}
\end{equation*}
Then
\begin{equation*}
\begin{aligned}
&\int_{\mathcal{O}}\nabla\beta(y+Y)\cdot\nabla\beta(y+Y)d\xi\\
\leq&2K\int_{\mathcal{O}} (\nabla\beta(y+Y)-\nabla \beta(Y))\cdot \nabla y d\xi+2K\int_{\mathcal{O}} \nabla \beta(Y)\cdot\nabla y d\xi+K^2\int_{\mathcal{O}} |\nabla Y|^2 d\xi.
\end{aligned}
\end{equation*}
Because of the left is nonnegative
\begin{equation}\label{4.18}
0\leq \int_{\mathcal{O}} (\nabla\beta(y+Y)-\nabla \beta(Y)+\epsilon\nabla y)\cdot \nabla y d\xi+\int_{\mathcal{O}} \nabla \beta(Y)\cdot\nabla y d\xi+\frac{K}{2}\int_{\mathcal{O}} |\nabla Y|^2 d\xi.
\end{equation}
We have $0\leq\langle \nabla(F+\epsilon I)(y),\nabla y \rangle_{2}+\langle -\Delta\beta(Y),y \rangle_{2}+\frac{K}{2}\int_{\mathcal{O}} |\nabla Y|^2 d\xi$.

Multiplying (\ref{asst}) by $y$, integrating over $\mathcal{O}$ and add $\epsilon(\langle -\Delta\beta(Y),y \rangle_{2}+\frac{K}{2}\int_{\mathcal{O}} |\nabla Y|^2 d\xi)$, one gets
\begin{equation*}
|y|^{2}_{2}\leq |x|_{2}|y|_{2}+\epsilon|y|_{2}|\Delta\beta(Y)|_{2}+\epsilon\frac{K}{2}\int_{\mathcal{O}} |\nabla Y|^2 d\xi.
\end{equation*}
Because for $a,b,c,d\geq0$, if $a^2\leq ab+ac+d^2$ we have $a\leq b+c+d$, so we have
\begin{equation*}
|y|_2\leq |x|_2+\epsilon |\Delta\beta(Y)|_2+\sqrt{\frac{\epsilon K}{2}}(\int_{\mathcal{O}} |\nabla Y|^2 d\xi)^{\frac{1}{2}},
\end{equation*}
such that (\ref{4.15}) is proved.

\vspace{10pt}

$\mathrm{iv})$ Finally, we prove that
\begin{equation*}
\begin{aligned}
\mathbb{E}\left[\int^{T}_{0}\| f(t) \|^{2}_{-1}dt\right]&=\mathbb{E}\left[\int^{T}_{0}\| \Delta\beta(Y(t))+b(t) \|^{2}_{-1}dt\right]\\
&\leq 2\mathbb{E}\left[\int^{T}_{0}\| \Delta\beta(Y(t)) \|^{2}_{-1}dt\right]+2\mathbb{E}\left[\int^{T}_{0}\| b(t) \|^{2}_{-1}dt\right]\\
&\leq 2K^2\mathbb{E}\left[\int^{T}_{0}\| Y(t) \|^{2}_{1}dt\right]+2\mathbb{E}\left[\int^{T}_{0}\| b(t) \|^{2}_{-1}dt\right]\\
&< \infty.
\end{aligned}
\end{equation*}

\end{proof}

\begin{remark}
For all $x\in H^{-1}$ and all $(t,\omega)\in [0,T]\times\Omega,\epsilon > 0$ we have
\begin{equation}\label{4.19}
\langle A_\epsilon(t,\omega,x),x\rangle_{-1}=\langle(F+\epsilon I)(J_\epsilon(x)),J_\epsilon(x)\rangle_2+\epsilon\|A_\epsilon(t,\omega,x)\|_{-1}^2,
\end{equation}
and for all $x\in L^{2}$,
\begin{equation}\label{4.20}
\langle A_\epsilon(t,\omega,x),x\rangle_2=\langle -\Delta(F+\epsilon I)(J_\epsilon(x)),J_\epsilon(x)\rangle_2+\epsilon|A_\epsilon(t,\omega,x)|_2^2.
\end{equation}
\end{remark}
We are going to show that $\{X_{\epsilon}\}$ is a Cauchy sequence in the space $L_{\mathrm{W}}^{2}(\Omega;C([0,T];H^{-1}))$. The idea of the proof is  based on a suitable a priori estimate for
\begin{equation*}
\mathbb{E}\left[\int_0^t\|A_\epsilon(X_\epsilon(s))\|_{-1}^2ds\right].
\end{equation*}

To this end, we will prove the following preliminary lemmas.
\begin{lemma}\label{lem3.3}
For all $x,y\in H^{-1}$, and $(t,\omega)\in[0,T]\times\Omega,\epsilon,\eta>0$, we have
\begin{equation}\label{4.21}
\begin{aligned}
\langle A_{\epsilon}(t,\omega,x)-A_{\eta}(t,\omega,y),x-y\rangle_{-1} =&\langle(F+\epsilon I)(J_{\epsilon}(x))-(F+\epsilon I)(J_{\eta}(y)),J_{\epsilon}(x)-J_{\eta}(y)\rangle_{2}  \\
& +(\epsilon-\eta)\langle J_{\eta}(y),J_{\epsilon}(x)-J_{\eta}(y)\rangle_{2}   \\
&+\langle A_{\epsilon}(x)-A_{\eta}(y),\epsilon A_{\epsilon}(x)-\eta A_{\eta}(y)\rangle_{-1}.
\end{aligned}
\end{equation}
\end{lemma}
\begin{proof}
We have
\begin{equation}\label{4.22}
\begin{aligned}
\langle A_{\epsilon}(t,\omega,x)-A_{\eta}(t,\omega,y),x-y\rangle_{-1} =&\langle A_{\epsilon}(x)-A_{\eta}(y),J_{\epsilon}(x)-J_{\eta}(y)\rangle_{-1}  \\
&+\langle A_{\epsilon}(x)-A_{\eta}(y),\epsilon A_{\epsilon}(x)-\eta A_{\eta}(y)\rangle_{-1}.
\end{aligned}
\end{equation}
On the other hand,
\begin{equation*}
\begin{aligned}
&\langle A_{\epsilon}(x)-A_{\eta}(y),J_{\epsilon}(x)-J_{\eta}(y)\rangle_{-1} \\
=&-\langle\Delta(F+\epsilon I)(J_{\epsilon}(x))-\Delta(F+\eta I)(J_{\eta}(y)),J_{\epsilon}(x)-J_{\eta}(y)\rangle_{-1} \\
=&-\langle\Delta(F+\epsilon I)(J_{\epsilon}(x))-\Delta(F+\epsilon I)(J_{\eta})(y),J_{\epsilon}(x)-J_{\eta}(y)\rangle_{-1} \\
&-(\epsilon-\eta)\langle\Delta J_{\eta}(y),J_{\epsilon}(x)-J_{\eta}(y)\rangle_{-1}.
\end{aligned}
\end{equation*}
Substituting this in (\ref{4.22}) yields (\ref{4.21}).

\end{proof}

\begin{lemma}\label{lem3.4}
Let $x\in H^{-1}$, $\epsilon >0$ and let $X_{\epsilon}$ be the solution to (\ref{4.1}). Then for any $T>0$ the following inequality holds here
\begin{equation}\label{4.23}
\begin{aligned}
&\mathbb{E}\left[\|X_\epsilon(t)\|_{-1}^2\right]+2\mathbb{E}\left[\int_0^t\langle J_\epsilon(X_\epsilon(s)),(\beta+\epsilon I)(J_\epsilon(X_\epsilon(s)))\rangle_2ds\right]+2\epsilon \mathbb{E}\left[\int_0^t\|A_\epsilon(X_\epsilon(s))\|_{-1}^2ds\right]\\
\leq &e^{(3K_1^2+1)t}(\|x\|_{-1}^2+3\sigma_1^2t+\mathbb{E}\left[\int^{t}_{0}\| f(s) \|^2_{-1}ds\right]+3K^2_1\mathbb{E}\left[\int_{0}^{t} \| Y(s) \|^2_{-1}ds\right]),\quad t\in[0,T].
\end{aligned}
\end{equation}
Moreover,
\begin{equation}\label{4.24}
\begin{aligned}
\mathbb{E}\left[\sup\limits_{t\in[0,T]}\|X_{\epsilon}(t)\|_{-1}^{2}\right]\leq&(2\| x \|^2_{-1}+6(2c^2_1+1)\sigma^2_{1}T+6(2c^2_1+1)K^2_1\mathbb{E}\left[\int_{0}^{T} \| Y(s) \|^2_{-1}ds\right]\\
&+2\mathbb{E}\left[\int^{T}_{0}\| f(s) \|^2_{-1}ds\right])e^{(12c^2_1K^2_1+6K^2_1)T}.
\end{aligned}
\end{equation}
\end{lemma}

\begin{proof}
By It\^{o}'s formula we have
\begin{equation}\label{4.25}
\begin{aligned}
d\|X_\epsilon(t)\|_{-1}^2=&-2\langle X_\epsilon(t),A_\epsilon(X_\epsilon(t))\rangle_{-1}dt+2\langle f(t),X_{\epsilon}(t) \rangle_{-1}dt\\
&+\mathrm{Tr}\left[\bar{\sigma}(X_\epsilon(t))\bar{\sigma}(X_\epsilon(t))^*\right]dt+2\langle X_\epsilon(t),\bar{\sigma}(X_\epsilon(t))dW(t)\rangle_{-1}.
\end{aligned}
\end{equation}
Owing to (\ref{4.21}), one has
\begin{equation}\label{4.26}
\begin{aligned}
&d\|X_\epsilon(t)\|_{-1}^2+2\langle(F+\epsilon I)(J_\epsilon(X_\epsilon(t))),J_\epsilon(X_\epsilon(t))\rangle_2dt+2\epsilon\|A_\epsilon(X_\epsilon(t))\|_{-1}^2dt\\
=&\text{Tr}[\sigma(X_\epsilon(t))\sigma(X_\epsilon(t))^*]dt+2\langle f(t),X_{\epsilon}(t) \rangle_{-1}dt+2\langle X_\epsilon(t),\sigma(X_\epsilon(t))dW(t)\rangle_{-1}.
\end{aligned}
\end{equation}
Integrating on $[0,t]$ and taking expectation we get that
\begin{equation}\label{4.27}
\begin{aligned}
&\mathbb{E}\left[\|X_\epsilon(t)\|_{-1}^2\right]+\mathbb{E}\left[\int^{t}_{0}\{2\langle(F+\epsilon I)(J_\epsilon(X_\epsilon(s))),J_\epsilon(X_\epsilon(s))\rangle_2+2\epsilon\|A_\epsilon(X_\epsilon(s))\|_{-1}^2\}ds\right]\\
=&\|x\|^2_{-1}+\mathbb{E}\left[\int^{t}_{0}\text{ Tr }[\sigma(X_\epsilon(s))\sigma(X_\epsilon(s))^*]ds+\int^{t}_{0}2\langle f(s),X_{\epsilon}(s) \rangle_{-1}ds\right].
\end{aligned}
\end{equation}
Keeping in mind the assumption (\ref{3.9}) we obtain that
\begin{equation}\label{4.28}
\begin{aligned}
&\mathbb{E}\left[\|X_\epsilon(t)\|_{-1}^2\right]+\mathbb{E}\left[\int^{t}_{0}\{2\langle(F+\epsilon I)(J_\epsilon(X_\epsilon(s))),J_\epsilon(X_\epsilon(s))\rangle_2+2\epsilon\|A_\epsilon(X_\epsilon(s))\|_{-1}^2\}ds\right]\\
\leq& \| x\|^{2}_{-1}+3\mathbb{E}\left[\int_{0}^{t}(\sigma_{1}^{2}+K^2_1(\| X_{\epsilon}(s) \|^2_{-1}+\| Y(s) \|^{2}_{-1}))ds\right]\\
&+2\mathbb{E}\left[\int^t_0\| X_{\epsilon}(s) \|^2_{-1}ds\right]^{\frac{1}{2}}\mathbb{E}\left[\int^t_0\| f(s) \|^2_{-1}ds\right]^{\frac{1}{2}}\\
\leq& \| x\|^{2}_{-1}+3\sigma^2_{1}t+3K^2_1\mathbb{E}\left[\int_{0}^{t} \| Y(s) \|^2_{-1}ds\right]\\
&+(3K^2_{1}+1)\mathbb{E}\left[\int^{t}_{0}\| X_{\epsilon}(s) \|^2_{-1}ds\right]+\mathbb{E}\left[\int^{t}_{0}\| f(s) \|^2_{-1}ds\right].
\end{aligned}
\end{equation}
From Gronwall's Lemma we get (\ref{4.23}). On the other hand going back to (\ref{4.26}), taking sup before taking expectation to get\\
\begin{equation}\label{4.29}
\begin{aligned}
\mathbb{E}\left[\sup_{t\in[0,T]}\|X_{\epsilon}(t)\|_{-1}^2\right]\leq& \| x\|^{2}_{-1}+3\mathbb{E}\left[\int_{0}^{t}(\sigma_{1}^{2}+K^2_1(\| X_{\epsilon}(s) \|^2_{-1}+\| Y(s) \|^{2}_{-1}))ds\right]+\mathbb{E}\left[\int^{T}_{0}\| f(s) \|^2_{-1}ds\right]\\
&+\mathbb{E}\left[\int^{T}_{0}\| X_{\epsilon}(s) \|^2_{-1}ds\right]+2\mathbb{E}\left[\sup_{r\in[0,T]}\left|\int_0^r\langle X_\epsilon(s),\bar{\sigma}(X_\epsilon(s))dW(s)\rangle_{-1}\right|\right].
\end{aligned}
\end{equation}
From Burkholder-Davis-Gundy inequality, we have
\begin{equation}\label{4.30}
\begin{aligned}
&\mathbb{E}\left[ \sup_{r\in[0,t]}\left|\int_{0}^{r}\langle X_{\epsilon}(s),\bar{\sigma}(X_{\epsilon}(s))dW(s)\rangle_{-1}\right|\right] \leq c_1\mathbb{E}\left[\left(\int_0^t|\bar{\sigma}(X_\epsilon(s))^*X_\epsilon(s)|_{-1}^2ds\right)^{1/2} \right]\\
\leq&\sqrt{3}c_{1}\mathbb{E}\left[\left(\int_{0}^{t}(\sigma_{1}^{2}+K^2_1\|Y(s)\|_{-1}^{2}+K_{1}^{2}\|X_{\epsilon}(s)\|_{-1}^{2})\|X_{\epsilon}(s)\|_{-1}^{2}ds\right)^{1/2}\right]\\
\leq&\sqrt{3}c_1\mathbb{E}\left[\left(\sup_{r\in[0,t]}\|X_\epsilon(r)\|_{-1}^2\int_0^t(\sigma_{1}^{2}+K^2_1\|Y(s)\|_{-1}^{2}+K_{1}^{2}\|X_{\epsilon}(s)\|_{-1}^{2})ds\right)^{1/2}\right] \\
\leq&\frac{1}{4}\mathbb{E}\left[\sup_{r\in[0,t]}\|X_{\epsilon}(r)\|_{-1}^{2}\right]+3c_{1}^{2}\mathbb{E}\left[\int_{0}^{t}(\sigma_{1}^{2}+K^2_1\|Y(s)\|_{-1}^{2}+K_{1}^{2}\|X_{\epsilon}(s)\|_{-1}^{2})ds\right].
\end{aligned}
\end{equation}
Going back to (\ref{4.29}) we have
\begin{equation}\label{4.31}
\begin{aligned}
\mathbb{E}\left[\sup_{t\in[0,T]}\|X_{\epsilon}(t)\|_{-1}^2\right]&\leq 2\| x \|^2_{-1}+6(2c^2_1+1)\sigma^2_{1}T+6(2c^2_1+1)K^2_1\mathbb{E}\left[\int_{0}^{T} \| Y(s) \|^2_{-1}ds\right]\\
&+2\mathbb{E}\left[\int^{T}_{0}\| f(s) \|^2_{-1}ds\right]+(12c^2_1K^2_1+6K^2_{1}+2)\mathbb{E}\left[\int^{T}_{0}\| X_{\epsilon}(s) \|^2_{-1}ds\right],
\end{aligned}
\end{equation}
by Gronwall's Lemma get (\ref{4.24}).
\end{proof}

\begin{lemma}\label{lem3.5}
For each $x\in L^{2}$, $T>0$ and each $\epsilon \in(0,\frac{1}{2\vee 3K})$ we have
\begin{equation}\label{4.32}
\begin{aligned}
&\mathbb{E}[|X_{\epsilon}(t)|_{2}^{2}]+\frac{1}{K}\mathbb{E}\left[  \int_{0}^{t}\int_{\mathcal{O}} |\nabla \beta(J_{\epsilon}(X_{\epsilon}(s))+Y)|^2 d\xi ds \right] \\
\leq&(|x|^2_2+3\sigma^2_{2}t+3K^2_2\mathbb{E}\left[\int_{0}^{t} | Y(s) |^2_{2}ds\right]+\mathbb{E}\left[\int^{t}_{0}| f(s) |^2_{2}ds\right]+2\mathbb{E}\left[  \int_{0}^{t} |\Delta\beta(Y)|^2_2 ds \right]\\
&+(K+1)\mathbb{E}\left[  \int_{0}^{t}|\nabla Y|^2_2 ds \right])e^{(3K_{2}^{2}+4)t},\quad\forall t\in[0,T].
\end{aligned}
\end{equation}
\begin{proof}
Applying It\^{o}'s formula to $|X_{\epsilon}(t)|_{2}^{2}$, yields
\begin{equation}\label{4.33}
\begin{aligned}
d|X_{\epsilon}(t)|_{2}^{2}=&-2\langle A_{\epsilon}(X_{\epsilon}(t)),X_{\epsilon}(t)\rangle_{2}dt+\|\bar{\sigma}(X_{\epsilon}(t))\|_{\mathcal{L}_{2}(H^{-1},L^{2})}^{2}dt  \\
&+2\langle f(t),X_{\epsilon}(t) \rangle dt+2\langle X_{\epsilon}(t)),\bar{\sigma}(X_{\epsilon}(t))dW(t)\rangle_{2}.
\end{aligned}
\end{equation}
By (\ref{4.20})
\begin{equation}\label{4.34}
\begin{aligned}
&d|X_{\epsilon}(t)|_{2}^{2}+2\langle-\Delta(F+\epsilon I)(J_\epsilon(X_\epsilon(t))),J_\epsilon(X_\epsilon(t))\rangle_2dt+2\epsilon|A_\epsilon(X_\epsilon(t))|_{2}^2dt\\
=&\|\bar{\sigma}(X_{\epsilon}(t))\|_{\mathcal{L}_{2}(H^{-1},L^{2})}^{2}dt+2\langle f(t),X_{\epsilon}(t) \rangle_{2}dt+2\langle X_{\epsilon}(t)),\bar{\sigma}(X_{\epsilon}(t))dW(t)\rangle_{2}.
\end{aligned}
\end{equation}
Integrating on $[0,T]$ and taking expectation, yields
\begin{equation}\label{4.35}
\begin{aligned}
&\mathbb{E}\left[ |X_{\epsilon}(t)|_{2}^{2} \right]+2\mathbb{E}\left[  \int_{0}^{t}\langle -\Delta(F+\epsilon I)(J_{\epsilon}(X_{\epsilon}(s))),J_{\epsilon}(X_{\epsilon}(s))\rangle_{2}ds \right]\\
\leq& |x|^2_2+\mathbb{E}\left[ \int_{0}^{t}\|\bar{\sigma}(X_{\epsilon}(s))\|_{\mathcal{L}_{2}(H^{-1},L^{2})}^{2}ds \right]+2\mathbb{E}\left[ \int^t_{0}\langle f(s),X_{\epsilon}(s) \rangle_{2} ds\right]\\
\leq&|x|^2_2+3\sigma^2_{2}t+3K^2_2\mathbb{E}\left[\int_{0}^{t} | Y(s) |^2_{2}ds\right]+\mathbb{E}\left[\int^{t}_{0}| f(s) |^2_{2}ds\right]+(3K^2_{2}+1)\mathbb{E}\left[\int^{t}_{0}| X_{\epsilon}(s) |^2_{2}ds\right].
\end{aligned}
\end{equation}
The item $$\mathbb{E}\left[  \int_{0}^{t}\langle -\Delta(F+\epsilon I)(J_{\epsilon}(X_{\epsilon}(s))),J_{\epsilon}(X_{\epsilon}(s))\rangle_{2}ds \right],$$ due to the presence of $Y$ is difficult to ensure is nonnegative, so by (\ref{4.18}), we add the $$2\mathbb{E}\left[  \int_{0}^{t}|\langle -\Delta\beta(Y),J_{\epsilon}(X_{\epsilon}(s)) \rangle_{2}|ds \right]+K\mathbb{E}\left[  \int_{0}^{t}|\nabla Y|^2 ds \right],$$ on both side of the above inequality.

\begin{equation*}
\begin{aligned}
&\mathbb{E}\left[ |X_{\epsilon}(t)|_{2}^{2} \right]+\frac{1}{K}\mathbb{E}\left[  \int_{0}^{t}\int_{\mathcal{O}} |\nabla \beta(J_{\epsilon}(X_{\epsilon}(s))+Y)|^2 d\xi ds \right]\\
\leq&\mathbb{E}\left[ |X_{\epsilon}(t)|_{2}^{2} \right]+2\mathbb{E}\left[  \int_{0}^{t}\langle -\Delta(F+\epsilon I)(J_{\epsilon}(X_{\epsilon}(s))),J_{\epsilon}(X_{\epsilon}(s))\rangle_{2}ds \right]\\
&+2\mathbb{E}\left[  \int_{0}^{t}|\langle -\Delta\beta(Y),J_{\epsilon}(X_{\epsilon}(s)) \rangle_{2}|ds \right]+K\mathbb{E}\left[  \int_{0}^{t}|\nabla Y|^2_2 ds \right]\\
\leq& |x|^2_2+3\sigma^2_{2}t+3K^2_2\mathbb{E}\left[\int_{0}^{t} | Y(s) |^2_{2}ds\right]+\mathbb{E}\left[\int^{t}_{0}| f(s) |^2_{2}ds\right]+(3K^2_{2}+1)\mathbb{E}\left[\int^{t}_{0}| X_{\epsilon}(s) |^2_{2}ds\right]\\
&+2\mathbb{E}\left[  \int_{0}^{t}|\langle -\Delta\beta(Y),J_{\epsilon}(X_{\epsilon}(s)) \rangle_{2}|ds \right]+K\mathbb{E}\left[  \int_{0}^{t}|\nabla Y|^2_2 ds \right]\\
\leq&|x|^2_2+3\sigma^2_{2}t+3K^2_2\mathbb{E}\left[\int_{0}^{t} | Y(s) |^2_{2}ds\right]+\mathbb{E}\left[\int^{t}_{0}| f(s) |^2_{2}ds\right]+(3K^2_{2}+1)\mathbb{E}\left[\int^{t}_{0}| X_{\epsilon}(s) |^2_{2}ds\right]\\
&+\mathbb{E}\left[  \int_{0}^{t} |\Delta\beta(Y)|^2_2 ds \right]+\mathbb{E}\left[  \int_{0}^{t} |J_{\epsilon}(X_{\epsilon}(s))|^2_2 ds \right]+K\mathbb{E}\left[  \int_{0}^{t}|\nabla Y|^2_2 ds \right]\\
\leq&|x|^2_2+3\sigma^2_{2}t+3K^2_2\mathbb{E}\left[\int_{0}^{t} | Y(s) |^2_{2}ds\right]+\mathbb{E}\left[\int^{t}_{0}| f(s) |^2_{2}ds\right]+(3K^2_{2}+1)\mathbb{E}\left[\int^{t}_{0}| X_{\epsilon}(s) |^2_{2}ds\right]\\
&+\mathbb{E}\left[  \int_{0}^{t} |\Delta\beta(Y)|^2_2 ds \right]+3\mathbb{E}\left[  \int_{0}^{t} |X_{\epsilon}(s)|^2_2 ds \right]+3\epsilon^2\mathbb{E}\left[  \int_{0}^{t} |\Delta\beta(Y)|^2_2 ds \right]\\
&+K(1+\frac{3\epsilon}{2})\mathbb{E}\left[  \int_{0}^{t}|\nabla Y|^2_2 ds \right].\\
\end{aligned}
\end{equation*}
Then we take $\epsilon$ small enough
\begin{equation*}
\begin{aligned}
&\mathbb{E}\left[ |X_{\epsilon}(t)|_{2}^{2} \right]+\frac{1}{K}\mathbb{E}\left[  \int_{0}^{t}\int_{\mathcal{O}} |\nabla \beta(J_{\epsilon}(X_{\epsilon}(s))+Y)|^2 d\xi ds \right]\\
\leq&|x|^2_2+3\sigma^2_{2}t+3K^2_2\mathbb{E}\left[\int_{0}^{t} | Y(s) |^2_{2}ds\right]+\mathbb{E}\left[\int^{t}_{0}| f(s) |^2_{2}ds\right]+(3K^2_{2}+4)\mathbb{E}\left[\int^{t}_{0}| X_{\epsilon}(s) |^2_{2}ds\right]\\
&+2\mathbb{E}\left[  \int_{0}^{t} |\Delta\beta(Y)|^2_2 ds \right]+(K+1)\mathbb{E}\left[  \int_{0}^{t}|\nabla Y|^2_2 ds \right],\\
\end{aligned}
\end{equation*}
by Gronwall inequality get (\ref{4.32}).
\end{proof}
\end{lemma}

\begin{lemma}\label{lem3.6}
Assume that Hypothesis 2 holds. Then for each $x\in L^2$ and each $\epsilon \in(0,\frac{1}{2\vee 3K})$ we have
\begin{equation}\label{4.36}
\begin{aligned}
\mathbb{E}\left[\int_{0}^{t}\|A_{\epsilon}(X_{\epsilon}(s))\|_{-1}^{2}ds\right]\leq& (2K+1)(|x|^2_2+3\sigma^2_{2}t+3K^2_2\mathbb{E}\left[\int_{0}^{t} | Y(s) |^2_{2}ds\right]+\mathbb{E}\left[\int^{t}_{0}| f(s) |^2_{2}ds\right]\\
&+2\mathbb{E}\left[  \int_{0}^{t} |\Delta\beta(Y)|^2_2 ds \right]+(2K+1)\mathbb{E}\left[  \int_{0}^{t}|\nabla Y|^2_2 ds \right])e^{(3K_{2}^{2}+4)t}.
\end{aligned}
\end{equation}
\end{lemma}
\begin{proof}
Note first that
\begin{equation}\label{4.37}
\begin{aligned}
\mathbb{E}\left[\int_{0}^{t}\|A_{\epsilon}(X_{\epsilon}(s))\|_{-1}^{2}ds\right] =&\mathbb{E}\left[\int_{0}^{t}\langle -\Delta(F+\epsilon I)(J_{\epsilon}(X_{\epsilon}(s))),(F+\epsilon I)(J_{\epsilon}(X_{\epsilon}(s)))\rangle_{2}ds\right]  \\
=&\mathbb{E}\left[\int_{0}^{t}\int_{\mathcal{O}} \nabla(F+\epsilon I)(J_{\epsilon}(X_{\epsilon}(s)))\cdot\nabla(F+\epsilon I)(J_{\epsilon}(X_{\epsilon}(s)))d\xi ds\right]\\
\leq& 2\mathbb{E}\left[\int_{0}^{t}\int_{\mathcal{O}} \nabla\beta(J_{\epsilon}(X_{\epsilon}(s))+Y)\cdot\nabla\beta(J_{\epsilon}(X_{\epsilon}(s))+Y)d\xi ds\right]\\
&+2\mathbb{E}\left[\int_{0}^{t}\int_{\mathcal{O}} |\nabla\beta(Y)|^2 d\xi+\epsilon\langle -\Delta(F+\epsilon I)(J_{\epsilon}(X_{\epsilon}(s))),J_{\epsilon}(X_{\epsilon}(s))\rangle_{2}ds\right].\\
\end{aligned}
\end{equation}
From the proof process of Lemma 4.4, we can know that
\begin{equation}\label{4.38}
\begin{aligned}
&\frac{1}{K}\mathbb{E}\left[  \int_{0}^{t}\int_{\mathcal{O}} |\nabla \beta(J_{\epsilon}(X_{\epsilon}(s))+Y)|^2 d\xi ds \right]\\
\leq&(|x|^2_2+3\sigma^2_{2}t+3K^2_2\mathbb{E}\left[\int_{0}^{t} | Y(s) |^2_{2}ds\right]+\mathbb{E}\left[\int^{t}_{0}| f(s) |^2_{2}ds\right]+2\mathbb{E}\left[  \int_{0}^{t} |\Delta\beta(Y)|^2_2 ds \right]\\
&+(K+1)\mathbb{E}\left[  \int_{0}^{t}|\nabla Y|^2_2 ds \right])e^{(3K_{2}^{2}+4)t},\quad\forall t\in[0,T].
\end{aligned}
\end{equation}
Therefore
\begin{equation}\label{4.39}
\begin{aligned}
&\mathbb{E}\left[\int_{0}^{t}\|A_{\epsilon}(X_{\epsilon}(s))\|_{-1}^{2}ds\right]\\
\leq& (2K+\epsilon)(|x|^2_2+3\sigma^2_{2}t+3K^2_2\mathbb{E}\left[\int_{0}^{t} | Y(s) |^2_{2}ds\right]+\mathbb{E}\left[\int^{t}_{0}| f(s) |^2_{2}ds\right]+2\mathbb{E}\left[  \int_{0}^{t} |\Delta\beta(Y)|^2_2 ds \right]\\
&+(K+1)\mathbb{E}\left[  \int_{0}^{t}|\nabla Y|^2_2 ds \right])e^{(3K_{2}^{2}+4)t}+2\mathbb{E}\left[\int_{0}^{t}\int_{\mathcal{O}} |\nabla\beta(Y)|^2 d\xi\right]\\
\leq& (2K+1)(|x|^2_2+3\sigma^2_{2}t+3K^2_2\mathbb{E}\left[\int_{0}^{t} | Y(s) |^2_{2}ds\right]+\mathbb{E}\left[\int^{t}_{0}| f(s) |^2_{2}ds\right]+2\mathbb{E}\left[  \int_{0}^{t} |\Delta\beta(Y)|^2_2 ds \right]\\
&+(2K+1)\mathbb{E}\left[  \int_{0}^{t}|\nabla Y|^2_2 ds \right])e^{(3K_{2}^{2}+4)t},
\end{aligned}
\end{equation}
and the proof is complete.
\end{proof}

\begin{theorem}\label{th3.1}
Assume that the above assumptions are fulfilled. Then for each $T>0$ and all $x\in L^2$ problem (\ref{3.1}) has a unique strong solution. Moreover, for $x\in H^{-1}$ problem (\ref{3.1}) has a unique generalized solution.
\end{theorem}

\begin{proof}

\textbf{Step I.}

For each $x\in L^2$ the sequence ${X_{\epsilon}}$ is Cauchy in $C_{W}([0,T];L^{2}(\Omega;H^{-1}))$.
Let $\epsilon$, $\eta>0$. Using It\^{o}'s formula for $\|X_{\epsilon}(t)-X_{\eta}(t)\|_{-1}^{2}$, yields
\begin{equation}\label{4.40}
\begin{aligned}
d\|X_{\epsilon}(t)-X_{\eta}(t)\|_{-1}^{2} =&-2\langle X_{\epsilon}(t)-X_{\eta}(t),A_{\epsilon}(X_{\epsilon}(t))-A_{\eta}(X_{\eta}(t))\rangle_{-1}dt  \\
&+2\langle X_{\epsilon}(t)-X_{\eta}(t),(\bar{\sigma}(X_{\epsilon}(t))-\bar{\sigma}(X_{\eta}(t))dW(t)\rangle_{-1} \\
&+\|\bar{\sigma}(X_{\epsilon}(t))-\bar{\sigma}(X_{\eta}(t)\|_{\mathcal{L}_{2}(H^{-1})}^{2}dt.
\end{aligned}
\end{equation}
Recalling identity (\ref{4.21}) we obtain that
\begin{equation}\label{4.41}
\begin{aligned}
&d\|X_{\epsilon}(t)-X_{\eta}(t)\|_{-1}^{2}+2\langle(F+\epsilon I)J_{\epsilon}(X_{\epsilon}(t))-(F+\epsilon I)J_\eta(X_\eta(t)),J_\epsilon(X_\epsilon(t))-J_\eta(X_\eta(t))\rangle_2dt \\
=&-2(\epsilon-\eta)\langle J_\eta(X_\eta(t)),J_\epsilon(X_\epsilon(t))-J_\eta(X_\eta(t))\rangle_2dt\\
&-2\langle A_\epsilon(X_\epsilon(t))-A_\eta(X_\eta(t)),\epsilon A_\epsilon(X_\epsilon(t))-\eta A_\eta(X_\eta(t))\rangle_{-1}dt \\
&+2\langle X_{\epsilon}(t)-X_{\eta}(t),(\bar{\sigma}(X_{\epsilon}(t))-\bar{\sigma}(X_{\eta}(t))dW(t)\rangle_{-1} +\|\bar{\sigma}(X_{\epsilon}(t))-\bar{\sigma}(X_{\eta}(t)\|_{\mathcal{L}_{2}(H^{-1})}^{2}dt.
\end{aligned}
\end{equation}
Integrating with respect to time and taking expectation
\begin{equation}\label{4.42}
\begin{aligned}
&\mathbb{E}[\|X_{\epsilon}(t)-X_{\eta}(t)\|_{-1}^{2}]\\
&+2\mathbb{E}\left[\int^{t}_{0} \langle(F+\epsilon I)J_{\epsilon}(X_{\epsilon}(s))-(F+\epsilon I)J_\eta(X_\eta(s)),J_\epsilon(X_\epsilon(s))-J_\eta(X_\eta(s))\rangle_2ds \right] \\
\leq& \mathbb{E}\left[K_1 \int^{t}_{0}\| X_{\epsilon}(s)-X_{\eta}(s)  \|]^2_{-1}ds\right]+2(\epsilon+\eta)\mathbb{E}\left[\int^{t}_{0} |J_{\eta}(X_{\eta}(s))|_{2}(|J_{\epsilon}(X_{\epsilon}(s))|_{2}+|J_{\eta}(X_{\eta}(s))|_{2})ds\right]\\
&+3(\epsilon+\eta)\mathbb{E}\left[\int^{t}_{0}(\| A_{\epsilon}(X_{\epsilon}(s)) \|^2_{-1}+\| A_{\eta}(X_{\eta}(s)) \|^2_{-1})ds\right],
\end{aligned}
\end{equation}
by (\ref{4.15}), (\ref{4.32}) and\ (\ref{4.36}),
\begin{equation}\label{4.43}
\begin{aligned}
&\mathbb{E}\left[\|X_{\epsilon}(t)-X_{\eta}(t)\|_{-1}^{2}\right]\\
&+2\mathbb{E}\left[\int^{t}_{0} \langle(F+\epsilon I)J_{\epsilon}(X_{\epsilon}(t))-(F+\epsilon I)J_\eta(X_\eta(t)),J_\epsilon(X_\epsilon(t))-J_\eta(X_\eta(t))\rangle_2dt \right] \\
\leq& (\epsilon+\eta)(12K+4T+6)(|x|^2_2+3\sigma^2_{2}T+3K^2_2\mathbb{E}\left[\int_{0}^{T} | Y(s) |^2_{2}ds\right]+\mathbb{E}\left[\int^{T}_{0}| f(s) |^2_{2}ds\right]\\
&+2\mathbb{E}\left[  \int_{0}^{T} |\Delta\beta(Y)|^2_2 ds \right]+(2K+1)\mathbb{E}\left[  \int_{0}^{T}|\nabla Y|^2_2 ds \right])e^{(3K_{2}^{2}+4)T}\\
&+4(\epsilon+\eta)\mathbb{E}\left[\int^{T}_{0} (2\epsilon^2|\Delta \beta(Y)|^2_2+\epsilon K|\nabla Y|^2_2) ds\right]+K_{1}\int^{t}_{0}\mathbb{E}\left[\| X_{\epsilon}(s)- X_{\eta}(s) \|^2_{-1}\right]ds.
\end{aligned}
\end{equation}
Take $\epsilon,\ \eta$ sufficiently small, then by Gronwall lemma
\begin{equation}\label{4.44}
\begin{aligned}
&\mathbb{E}[\|X_{\epsilon}(t)-X_{\eta}(t)\|_{-1}^{2}]\\
&+2\mathbb{E}\left[\int^{t}_{0} \langle(F+\epsilon I)J_{\epsilon}(X_{\epsilon}(t))-(F+\epsilon I)J_\eta(X_\eta(t)),J_\epsilon(X_\epsilon(t))-J_\eta(X_\eta(t))\rangle_2dt \right] \\
\leq& (\epsilon+\eta)(12K+4T+6)(|x|^2_2+3\sigma^2_{2}T+3K^2_2\mathbb{E}\left[\int_{0}^{T} | Y(s) |^2_{2}ds\right]+\mathbb{E}\left[\int^{T}_{0}| f(s) |^2_{2}ds\right]\\
&+3\mathbb{E}\left[  \int_{0}^{T} |\Delta\beta(Y)|^2_2 ds \right]+(2K+2)\mathbb{E}\left[  \int_{0}^{T}|\nabla Y|^2_2 ds \right])e^{(3K_{2}^{2}+4)T+K_1t}.
\end{aligned}
\end{equation}
Therefore, the sequence ${X_{\epsilon}}$ is Cauchy in $C_W([0,T];L^2(\Omega;H^{-1}))$ as claimed.\\

\textbf{Step II.}

For each $x\in L^2$ the sequence ${X_{\epsilon}}$ is Cauchy in $L_{W}^{2}(\Omega;C([0,T];H^{-1}))$.
We first note that by (\ref{4.32}), (\ref{4.36}) and (\ref{4.41}), same as Step 1, we deduce\\
\begin{equation}\label{4.45}
\begin{aligned}
&\mathbb{E}[\sup\limits_{t\in[0,T]}\|X_\epsilon(t)-X_\eta(t)\|_{-1}^2]\\
\leq& (\epsilon+\eta)(24K+8T+12)(|x|^2_2+3\sigma^2_{2}T+3K^2_2\mathbb{E}\left[\int_{0}^{T} | Y(s) |^2_{2}ds\right]+\mathbb{E}\left[\int^{T}_{0}| f(s) |^2_{2}ds\right]\\
&+3\mathbb{E}\left[  \int_{0}^{T} |\Delta\beta(Y)|^2_2 ds \right]+(2K+2)\mathbb{E}\left[  \int_{0}^{T}|\nabla Y|^2_2 ds \right])e^{(3K_{2}^{2}+2K_1+4c^2_1K^2_1+4)T}.
\end{aligned}
\end{equation}
By using Burkholder-Davis-Gundy inequality
\begin{equation}\label{4.46}
\begin{aligned}
&\mathbb{E}\left[\sup_{t\in[0,T]}\int_{0}^{t}\langle X_{\epsilon}(s)-X_{\eta}(s),(\bar{\sigma}(X_{\epsilon}(s))-\bar{\sigma}(X_{\eta}(s))dW(s)\rangle_{-1}\right]\\
\leq& c_{1}\mathbb{E}\left[\left(\int_{0}^{T}|(\bar{\sigma}(X_{\epsilon}(s))-\bar{\sigma}(X_{\eta}(s)))^{*}(X_{\epsilon}(s)-X_{\eta}(s))|_{-1}^{2}ds\right)^{1/2}\right]  \\
\leq& c_1K_1 \mathbb{E}\left[\left(\int_0^T\|X_\epsilon(s)-X_\eta(s)\|_{-1}^4ds\right)^{1/2}\right] \\
\leq& c_1K_1 \mathbb{E}\left[\left(\sup_{t\in[0,T]}\|X_\epsilon(t)-X_\eta(t)\|_{-1}^2\int_0^T\|X_\epsilon(s)-X_\eta(s)\|_{-1}^2)ds\right)^{1/2}\right] \\
\leq&\frac{1}{4}\mathbb{E}\left[\sup_{t\in[0,T]}\|X_{\epsilon}(t)-X_{\eta}(t)\|_{-1}^{2}\right]+c_{1}^{2}K_{1}^{2}\mathbb{E}\left[\int_{0}^{T}\|X_{\epsilon}(s)-X_{\eta}(s)\|_{-1}^{2})ds\right],
\end{aligned}
\end{equation}
and by Gronwall lemma in (\ref{4.45}) we obtain that
\begin{equation}\label{4.47}
\begin{aligned}
&\mathbb{E}\left[\sup\limits_{t\in[0,T]}\|X_\epsilon(t)-X_\eta(t)\|_{-1}^2\right]\\
\leq& (\epsilon+\eta)(24K+8T+12)(|x|^2_2+3\sigma^2_{2}T+3K^2_2\mathbb{E}\left[\int_{0}^{T} | Y(s) |^2_{2}ds\right]+\mathbb{E}\left[\int^{T}_{0}| f(s) |^2_{2}ds\right]\\
&+3\mathbb{E}\left[  \int_{0}^{T} |\Delta\beta(Y)|^2_2 ds \right]+(2K+2)\mathbb{E}\left[  \int_{0}^{T}|\nabla Y|^2_2 ds \right])e^{(3K_{2}^{2}+2K_1+4c^2_1K^2_1+4)T}.
\end{aligned}
\end{equation}
\textbf{Step III.}

Existence of a strong solution.

Since $\{X_{\epsilon}\}$ is Cauchy in $L_W^2(\Omega;C([0,T];H^{-1}))$ there exists $X\in L_{W}^{2}(\Omega;C([0,T];H^{-1}))$ such that
\begin{equation}\label{4.48}
\lim_{\epsilon\to0}X_\epsilon=\bar{X}\quad\text{in}\ L_W^2(\Omega;C([0,T];H^{-1})).
\end{equation}
Because of $||x_{\epsilon}(t)-J_{\epsilon}(x_{\epsilon}t)||_{-1}=\epsilon||A_{\epsilon}(x_{\epsilon}(t))||_{-1}$, we have
\begin{equation}\label{4.49}
\lim\limits_{\epsilon\to0}J_\epsilon(X_\epsilon)=\bar{X}\quad\text{in}\ L^2(0,T;L^2(\Omega;H^{-1})).
\end{equation}
By using (\ref{4.36})
\begin{equation}\label{4.50}
\begin{aligned}
&\mathbb{E}\left[\int_{0}^{T}\|X_{\epsilon}(t)-J_{\epsilon}(X_{\epsilon}(t))\|_{-1}^{2}dt\right]\\
\leq&\epsilon^2 (2K+1)(|x|^2_2+3\sigma^2_{2}T+3K^2_2\mathbb{E}\left[\int_{0}^{T} | Y(s) |^2_{2}ds\right]+\mathbb{E}\left[\int^{T}_{0}| f(s) |^2_{2}ds\right]\\
&+2\mathbb{E}\left[  \int_{0}^{T} |\Delta\beta(Y)|^2_2 ds \right]+(2K+1)\mathbb{E}\left[  \int_{0}^{T}|\nabla Y|^2_2 ds \right])e^{(3K_{2}^{2}+4)T}.
\end{aligned}
\end{equation}
Since $F$ is Lipschitzan and therefore $F^{-1}$ is strongly monotone, it follows that
\begin{equation}\label{4.51}
\begin{aligned}
&\mathbb{E}\left[\int^{t}_{0}\| F(J_{\epsilon}(X_{\epsilon}))-F(J_{\eta}(X_{\eta}))\|^2_{2}dt\right]\\
\leq& K\mathbb{E}\left[\int^{t}_{0}\langle F(J_{\epsilon}(X_{\epsilon}))-F(J_{\eta}(X_{\eta})) , J_{\epsilon}(X_{\epsilon})-J_{\eta}(X_{\eta})\rangle_2 dt\right]\\
\leq& K(\epsilon+\eta)(12K+4T+6)(|x|^2_2+3\sigma^2_{2}T+3K^2_2\mathbb{E}\left[\int_{0}^{T} | Y(s) |^2_{2}ds\right]+\mathbb{E}\left[\int^{T}_{0}| f(s) |^2_{2}ds\right]\\
&+3\mathbb{E}\left[  \int_{0}^{T} |\Delta\beta(Y)|^2_2 ds \right]+(2K+2)\mathbb{E}\left[  \int_{0}^{T}|\nabla Y|^2_2 ds \right])e^{(3K_{2}^{2}+4)T+K_1t}.
\end{aligned}
\end{equation}
So ${F(J_{\epsilon}(X_{\epsilon}))}$ is Cauchy in $L^{2}((0,T)\times \Omega \times \mathcal{O})$. Therefore there exists $Z \in L^{2}((0,T)\times\Omega\times\mathcal{O})$ such that
\begin{equation}\label{4.52}
\lim\limits_{\epsilon\to0}F(J_\epsilon(X_\epsilon))=Z\quad\text{in}\ L^2((0,T)\times\Omega\times\mathcal{O}).
\end{equation}
On the other hand, by estimate (\ref{4.32}) we know that  is bounded in $L^2((0,T)\times\Omega\times\mathcal{O})$ and therefore along a subsequence, again noted ${\epsilon}$, we have
\begin{equation}\label{4.53}
\lim_{\epsilon\to0}X_\epsilon=\bar{X}\quad\text{weakly in}\ L^2(\Omega\times(0,T)\times\mathcal{O}).
\end{equation}
Since
\begin{equation}\label{4.54}
|J_\epsilon(X_\epsilon(t))|_2\leq|X_\epsilon(t)|_2+\epsilon |\Delta\beta(Y)|_2+\sqrt{\frac{\epsilon K}{2}}\|Y\|_{H^1},\quad\forall t\in[0,T],
\end{equation}
from (\ref{4.51}) and (\ref{4.54})
\begin{equation}\label{4.55}
\lim\limits_{\epsilon\to0}J_\epsilon(X_\epsilon)=\bar{X}\quad\text{weakly in}\ L^2(\Omega\times(0,T)\times\mathcal{O}).
\end{equation}
Since the map $y \rightarrow F(y)$ is maximal monotone on $L^{2}(\Omega\times(0,T)\times\mathcal{O})$, it is weakly-strongly closed and so $Z=F(X)$ in $\Omega\times(0,T)\times\mathcal{O}$.\\
Finally, since $\{(\beta+\epsilon I)(J_{\epsilon}(X_{\epsilon}))\}$ is bounded in $L^2([0,T];L^2(\Omega,H_0^1))$ (because $A_{\epsilon}(X_{\epsilon})=\Delta(\beta+\epsilon I) (J_{\epsilon}(X_{\epsilon}))$ is bounded in $L^{2}([0,T];L^{2}(\Omega,H^{-1}))$), we can conclude that there exists a sequence ${\epsilon_{k}}$ convergent to 0 such that
\begin{equation}\label{4.56}
\lim\limits_{k\to\infty}(F+\epsilon_kI)(J_{\epsilon_k}(X_{\epsilon_k}))\to Z=F(\bar{X})\quad\text{weakly in}\ L^2([0,T]\times\Omega;H_0^1).
\end{equation}
So, letting $k\rightarrow \infty$ in the identity
\begin{equation}\label{4.57}
X_{\epsilon_k}(t)=x+\int_0^t\Delta(F+\epsilon_kI)(J_{\epsilon_k}(X_{\epsilon_k}(s)))ds+\int^{t}_{0}f(t)dt+\bar{\sigma}(X_{\epsilon_k}(t))dW(t).
\end{equation}
We get that $\bar{X}$ is a strong solution to (\ref{3.8}). \\

\textbf{Step IV.}

Let $x\in H^{-1}$ and let $x_{n}$ be a sequence in $L^{2}$ convergent to $x$ in $H^{-1}$.

Denote by $X_{n}$ the strong solution to (\ref{3.8}) with $x_{n}$ replacing $x$. Then for any $m,\ n\in \mathbb{N}$ we have by It\^{o}'s formula
\begin{equation}\label{4.58}
\begin{aligned}
d\|X_{m}(t) -X_{n}(t))\|_{-1}^{2}=&2\langle \Delta F(X_{m}(t))-\Delta F(X_{n}(t)),X_{m}(t)-X_{n}(t))\rangle_{-1}dt  \\
&+2\langle X_m(t)-X_n(t),(\bar{\sigma}(X_m(t))-\bar{\sigma}(X_n(t)))dW(t)\rangle_{-1} \\
&+\|\bar{\sigma}(X_{m}(t))-\bar{\sigma}(X_{n}(t))\|_{\mathcal{L}_{2}(H^{-1})}^{2}dt.
\end{aligned}
\end{equation}
Because of $\langle \Delta F(X_{m}(t))-\Delta F(X_{n}(t)),X_{m}(t)-X_{n}(t)) \rangle\leq 0$, integrating with respect to t
\begin{equation}\label{4.59}
\begin{aligned}
\|X_{m}(t)-X_{n}(t))\|_{-1}^{2}\leq&2\int_{0}^{t}\langle X_{m}(s)-X_{n}(s),(\bar{\sigma}(X_{m}(s))-\bar{\sigma}(X_{n}(s)))dW(s)\rangle_{-1}\\
&+\|x_{n}-x_{m}\|_{-1}^{2}+\int_{0}^{t}\|\bar{\sigma}(X_{m}(s))-\bar{\sigma}(X_{n}(s))\|_{\mathcal{L}_{2}(H^{-1})}^{2}ds.
\end{aligned}
\end{equation}
we take the supremum on $t$ and take expectation
\begin{equation}\label{4.60}
\begin{aligned}
&\mathbb{E}\left[\sup_{t\in[0,T]}\|X_{m}(t)-X_{n}(t))\|_{-1}^{2}\right]\\ \leq&\|x_{n}-x_{m}\|_{-1}^{2}+K_1\mathbb{E}\left[\int_0^T\|X_m(s))-X_n(s))\|_{-1}^2ds\right]\\
&+2\mathbb{E}\left[\sup_{t\in[0,T]}\int_0^T\langle X_m(s)-X_n(s),(\bar{\sigma}(X_m(s))-\bar{\sigma}(X_n(s)))dW(s)\rangle_{-1}\right].
\end{aligned}
\end{equation}
 Using the same way as before (\ref{4.29}), we can conclude that the sequence $\{X_{n}\}_{n\geq 1 }$ is Cauchy in\\ $L_W^2(\Omega;C([0,T];H^{-1}))$ and consequently it is convergent to a generalized solution of (\ref{4.1}). Then we consider take the solution of (\ref{3.8}) back to (\ref{3.1}). Because of the assumption of $Y$, we have
\begin{equation}\label{4.61}
X(t)=x+\int_0^t\Delta\beta(X(s))ds+\int_0^t\sigma(X(s))dW(s),\quad\forall t\in[0,T],\ \mathbb{P}\text{-a.s.},
\end{equation}
\begin{equation}\label{4.62}
X(t)=\bar{X}(t)+Y(t)\in L_W^2(\Omega;C([0,T];H^{-1}))\cap L^2(0,T;L^2(\Omega;L^2)),
\end{equation}
\begin{equation}\label{4.63}
\begin{aligned}
||\beta(X)||_{2}=||\beta(\bar{X}+Y)||_{2}&\leq|| \beta(\bar{X})||_{2}+||\beta(\bar{X}+Y)-\beta(\bar{X})||_{2}\\
&\leq || \beta(\bar{X})||_{2}+K||Y||_{2},
\end{aligned}
\end{equation}
from previous estimates and assumptions easy to get $\beta(X)\in L^2(0,T;L^2(\Omega;H^1))$.\\
So $X(t)$ is the strong solution of (\ref{3.1}).
When $x\in H^{-1}$ by using Step IV, we have that equation (\ref{3.1}) have a generalized solution.
\end{proof}

\noindent \textbf{Acknowledgements}
D.G. aknowledges financial support from National Sciences and Engineering Research Council (NSERC), Canada,
Grant/Award Number: RGPIN-2025-03963.

\end{document}